\documentclass[]{gSTA2e}
\usepackage{epsfig}

\usepackage{amssymb,amsmath,amsthm,url}
\usepackage{multirow}
\usepackage{moreenum}

\usepackage{hyperref}
\usepackage{setspace}
\usepackage{dsfont}
\usepackage[labelformat=simple]{subcaption}

\usepackage{enumitem}

\usepackage{times,amsfonts,eucal}
\usepackage{graphicx,ifthen}
\usepackage{wrapfig}
\usepackage{latexsym}
\usepackage{color}
\usepackage{mathrsfs}
\usepackage{bm}
\usepackage{srctex}

\newcommand{\mylabel}[2]{#2\def\@currentlabel{#2}\label{#1}}

\usepackage{mathtools}

\DeclarePairedDelimiter\ceil{ \lceil}{ \rceil}
\theoremstyle{plain}
\newtheorem{theorem}{Theorem}[section]
\newtheorem{corollary}[theorem]{Corollary}
\newtheorem{lemma}[theorem]{Lemma}
\newtheorem{proposition}[theorem]{Proposition}
\newtheorem{condition}[theorem]{Condition}
\theoremstyle{definition}

\newtheorem{example}[theorem]{Example}
\newtheorem{notation}[theorem]{Notation}
\theoremstyle{remark}

\numberwithin{equation}{section}

\newcommand{\N}{\mathbb{N}}
\newcommand{\R}{\mathbb{R}}

\newcommand{\Z}{\mathbb{Z}}

\newcommand{\p}{\mathbb{P}}

\newcommand{\C}{\mathbb{C}}
\newcommand{\cA}{\mathcal{A}}
\newcommand{\cB}{\mathcal{B}}

\newcommand{\cG}{\mathcal{G}}
\newcommand{\cF}{\mathcal{F}}
\newcommand{\cH}{\mathcal{H}}

\newcommand{\cL}{\mathcal{L}}
\newcommand{\cS}{\mathcal{S}}
\newcommand{\cT}{\mathcal{T}}
\newcommand{\cN}{\mathcal{N}}

\newcommand{\cO}{\mathcal{O}}
\newcommand{\cV}{\mathcal{V}}

\newcommand{\fS}{\mathfrak{S}}

\newcommand{\E}[1]{\mathbb{E}\left [ \, #1 \, \right ]}
\newcommand{\Ec}[1]{\mathbb{E}^*\left [ \, #1 \, \right ]}

\renewcommand{\epsilon}{\varepsilon}
\renewcommand{\phi}{\varphi}

\newcommand{\pspace}{(\Omega,\cA,\p)}

\newcommand{\intd}[1]{\mathrm{d}#1}

\newcommand{\norm}[1]{\left\lVert #1 \right\rVert}
\newcommand{\scalar}[2]{\left\langle #1,#2 \right\rangle}

\newcommand{\argmin}[1]{\operatorname*{arg\,min}_{#1}\,}
\newcommand{\dzn}[1]{d_{\infty}\left( #1 \right) }
\newcommand{\darg}{\,\cdot\,}
\newcommand{\supp}{\text{supp}\,}
\newcommand{\1}[1]{\,\mathds{1}\! \left\{ #1 \right\} }

\allowdisplaybreaks

\begin{document}

\title{Non-parametric Regression for Spatially Dependent Data with Wavelets\textsuperscript{a}
\footnote{\textsuperscript{a}This research was supported by the German Research Foundation (DFG), Grant Number KR-4977/1 and by the Fraunhofer ITWM, 67663 Kaiserslautern, Germany.}
}

\author{
\name{Johannes T. N. Krebs\textsuperscript{b}\thanks{CONTACT: jtkrebs@ucdavis.edu} } 
\affil{\textsuperscript{b}Department of Statistics, University of California, Davis, CA, 95616, USA}
}

\maketitle

\begin{abstract}
We study non-parametric regression estimates for random fields. The data satisfies certain strong mixing conditions and is defined on the regular $N$-dimensional lattice structure. We show consistency and obtain rates of convergence. The rates are optimal modulo a logarithmic factor in some cases. As an application, we estimate the regression function with multidimensional wavelets which are not necessarily isotropic. We simulate random fields on planar graphs with the concept of concliques (cf. \cite{kaiser2012}) in numerical examples of the estimation procedure.
\end{abstract}


\begin{classcode}Primary: 62G08, 62H11, 65T60; Secondary: 65C40, 60G60.\end{classcode}

\section{Introduction}
In this article we study a non-parametric regression model with random design for data which is observed on a spatial structure such as a regular $N$-dimensional lattice or a finite and undirected graph $G=(V,E)$ with a set of nodes $V$ and a set of edges $E$. Consider the random field $(X,Y)=\{ (X(s),Y(s)): s\in \Z^N \} \subseteq \R^d\times\R$. We assume that $(X,Y)$ has equal marginal distributions, e.g., $(X,Y)$ is stationary. Denote the probability distribution of the $X(s)$ by $\mu_X$. The process satisfies the regression model
\begin{align}\label{lsqI}
Y(s) = m( X(s) ) + \varsigma( X(s) ) \, \epsilon(s), \quad s\in \Z^N,
\end{align}
where $m$ and $\varsigma$ are two elements of the function space $L^2(\mu_X)$. The collection of error terms $\epsilon=\{\epsilon(s):s\in \Z^N \}$ is independent of $X$. The $\epsilon(s)$ have mean zero and unit variance. There is a vast literature on non-parametric regression models, see, e.g., \cite{hardle1990applied}, \cite{gyorfi} and \cite{gyorfi2013nonparametric}. A particular choice for the estimation of $m$ and $\varsigma$ are sieve estimators, see \cite{grenander1981abstract}. One class of sieve estimators are neural networks: \cite{hornik1991approximation} investigates approximation properties of multilayer feedforward networks. Rates of $L^2$-convergence for sigmoidal neural networks have been studied by \cite{barron1994approximation} and \cite{mccaffrey1994convergence}. \cite{franke06} use neural networks for modelling financial time series. \cite{kirch2014uniform} model  autoregressive processes by a feedforward neural network.

Another popular choice for the construction of the sieve are wavelets, see \cite{hardle2012wavelets} and \cite{fan1996local}. In this article, we consider the sieve estimator as defined in \cite{gyorfi} and we construct the sieve in applications with general multidimensional wavelets. The wavelet method has already been studied both in the classical i.i.d.\ case and for dependent data in various ways: \cite{donoho1996density} and \cite{donoho1998minimax} use wavelets for univariate density estimation with i.i.d.\ data. \cite{cai1999adaptive} studies block thresholding of the wavelet estimator in the regression model with fixed design. \cite{kerkyacharian2004regression} construct warped wavelets for the random design regression model which admit an orthonormal basis w.r.t.\ the design distribution. \cite{kulik2009wavelet} use warped wavelets in the regression model with dependent data and heteroscedastic error terms. \cite{brown2010nonparametric} study the wavelet method in the context of non-parametric regression estimators for exponential families.

Recently, the analysis of spatial data has gained importance in many applications, e.g., in astronomy, image analysis, environmental sciences or more general in GIS applications. The monographs of \cite{cressie1993statistics} and \cite{kindermann1980markov} offer a detailed introduction to this topic. Non-parametric regression models (with random design) for dependent data are a major tool in spatial statistics. We only mention a few related references: \cite{koul1977behavior}, \cite{roussas1992asymptotic}, \cite{baraud2001adaptive}, \cite{guessoum2010kernel}, \cite{yahia2012nonlinear}, \cite{li2016wavelet}.
 
So far, the kernel method has been popular when considering regression models for spatial data, see, e.g., \cite{carbon2007kernel} or \cite{hallin2004local}. The kernel method is an efficient tool if the design distribution has unbounded support. However, it can have disadvantages if the design distribution is compactly supported. In this case, the results can suffer from a boundary bias. Moreover, the kernel method requires a smooth regression function, e.g., two-times continuous differentiability.

In situations where these requirements are not satisfied, the wavelet method is an alternative which performs relatively well because of its extraordinary adaptability to local irregularities (e.g., jump discontinuities) of the underlying regression function, see also \cite{hall1995formulae} or \cite{gyorfi}. So smoothness conditions are only necessary in a piecewise sense. In particular, (hard thresholding) wavelet estimates can achieve a nearly optimal rate in the minimax sense for a variety of function spaces such as Besov or H{\"o}lder spaces.

However, the wavelet method has received little attention: \cite{li2016nonparametric} studies a wavelet estimator for the non-parametric regression model in the context of spatially dependent data under the assumption that the design distribution of the $X(s)$ is known. In this article, we continue with these ideas but we remove the assumption that the design distribution is known. We transfer the non-parametric regression model of \cite{gyorfi} for i.i.d.\ data to spatially dependent data. The model of \cite{gyorfi} has three important features. Firstly, the regression function $m$ can be any function in $L^2(\mu_X)$. It is not required that $m$ belongs to a certain range of function classes. E.g., other papers in the wavelet context often assume that the regression function belongs to the class of Besov spaces. Secondly, the function classes we construct the estimator from can take a very general form; we could use neural networks instead of multidimensional wavelets. Thirdly, the predicted variables $Y(s)$ are not necessarily bounded and neither the design distribution of the $X(s)$ nor the distribution of the error terms $\epsilon(s)$ needs to admit a density w.r.t.\ the Lebesgue measure. Furthermore, in this paper, we enrich the model with the following novelties. The data is not necessarily i.i.d.\ distributed any more. We prove consistency and derive rates of convergence of the least-squares estimator under strong mixing conditions. We relax the assumptions on the marginal distributions of the random field $(X,Y)$: the design distribution does not have to be known and does not have to admit a density w.r.t.\ the $d$-dimensional Lebesgue measure. The latter condition is assumed for instance in \cite{hallin2004local}. In applications we choose $d$-dimensional wavelets to construct the sieve. These wavelets can take a very general form and do not have to be isotropic.

Moreover, we remove the usual assumption of stationarity: we show that our estimator is consistent if the random field has equal marginal distributions. This is useful in applications to (Markov) random fields defined on irregular graphical networks which do not satisfy the usual definitions of stationarity. A Gaussian random field defined on a finite graph $G=(V,E)$ is such an example. There, the dependency structure of the data is determined by the adjacency matrix of $G$ and is supposed to vanish with an increasing graph-distance. Particular applications we have in mind are data like traffic intensity or road roughness indices on road networks, which may be represented on graphs.

The simulation examples in the present manuscript are constructed with the algorithm of \cite{kaiser2012} and use the concept of concliques. This approach puts us in position to consider our simulation as iterations of an ergodic Markov chain and we achieve a fast convergence of the simulated random field when compared to the Gibbs sampler. We give two simulation examples where we consider one bivariate and one univariate non-parametric linear regression problem on real graphical structures. The results give encouraging prospects in the handling of random fields on graphs.

Altogether, on the one hand, the main contribution of the paper is the generalization of the theory of distribution-free non-parametric regression of \cite{gyorfi} to spatially dependent data. On the other hand, we demonstrate how practical inference on irregular graphs can be performed with the studied estimation technique.

The remainder of this manuscript is organized as follows: we introduce the basic notation in detail in Section~\ref{Section_NonParaRegGeneral}. Besides, we present two general theorems on the consistency and the rate of convergence of the truncated non-parametric linear least-squares estimator. In Section~\ref{Section_NonParaRegWavelets} we use general $d$-dimensional wavelets to construct a consistent estimator of the regression function. Additionally, we derive rates of convergence for this estimator in examples where the regression function satisfies certain smoothness conditions. Section~\ref{ExamplesOfApplication} is devoted to numerical applications: we present simulation concepts for random fields on graphical structures and discuss the developed theory in two examples. Section~\ref{Section_Proofs} contains the proofs of the presented theorems. Appendix~\ref{Section_ExponentialInequalities} consists of useful exponential inequalities for dependent sums. Appendix~\ref{Appendix_Example} contains a deferred verification of an example.

\section{Regression Estimation for Spatially Dependent Data}\label{Section_NonParaRegGeneral}

In this section we present the main results of this article: consistency properties of the proposed estimators and their rates of convergence.

\subsection{Notation and Definitions}
We work on a probability space $\pspace$ that is equipped with a generic random field $Z$. In our application $Z$ will often be the random field $(X,Y)$. So $Z$ is a collection of random variables $\{ Z(s): s\in \Z^N \}$, where $N$ is the lattice dimension. Each $Z(s)$ maps from $\Omega$ to $S$, where $(S,\fS)$ is a measurable space.

The random field is called (strictly) stationary if for each $k \in \N_+$, for all points $s_1,\ldots,s_k \in \Z^N$ and for each translation $t\in \Z^N$, the joint distribution of $\{ Z(s_1+t),\ldots,Z(s_k+t) \}$ coincides with the joint distribution of $\{ Z(s_1),\ldots,Z(s_k) \}$.

Furthermore, if $j\in\N$ and $A\in\R_+$, we write $2^j \simeq A$ if and only if $2^j \le A < 2^{j+1}$.

If $U$ is a random variable on $\pspace$ with values in $[-\infty,\infty]$, we write $\norm{U}_{\p,p}$ for the $p$-norm of $U$ w.r.t.\ $\p$ for $p\in[1,\infty]$, i.e., $\norm{U}_{\p,p}^p = \E{|U|^p}$. Similarly, if $\nu$ is a measure on $(\R^d,\cB(\R^d))$ and $f$ is a real-valued function on $\R^d$, we write $\norm{f}_{L^p(\nu)}$ for the $L^p$-norm of $f$ w.r.t.\ $\nu$ for $p\in[1,\infty]$.

Write $\lambda$ (resp. $\lambda^d$) for the one-dimensional (resp. $d$-dimensional) Lebesgue measure on $(\R,\cB(\R))$ (resp. $(\R^d,\cB(\R^d))$) and denote the space of square integrable Borel functions on $\R^d$ w.r.t.\ the $d$-dimensional Lebesgue measure by $L^2\left(\R^d, \cB(\R^d), \lambda^d \right)$. We sometimes abbreviate it also by $L^2(\lambda^d)$.  

We define the $2$-norm of a square matrix $A=(a_{i,j})_{1\le i,j\le d}\in\R^{d\times d}$ as $\norm{A}_{2} = \sup_{x: \norm{x}_2 = 1} \norm{Ax}_2$, where $\norm{x}_2$ is the Euclidean 2-norm of $x\in\R^d$.

Next, we consider the lattice $\Z^N$. We write $\norm{\,\cdot\,}_{\infty}$ for the maximum norm on $\R^N$ and $d_{\infty}$ for the corresponding metric which can be extended to non-empty subsets $I,J$ of $\Z^N$ via $d_{\infty}(I,J) \coloneqq \min \{ d_{\infty}(s,t): s\in I, t\in J \}$. Additionally, we write $s \le t$ for $s,t\in \R^N$ if and only if the single coordinates satisfy $s_i \le t_i$ for each $1\le i \le N$. The $N$-dimensional vector $(1,\ldots,1)$ is abbreviated by $e_N$.

Let $I$ be a subset of $\Z^N$, the $\sigma$-algebra which is generated by the $Z(s)$ with $s$ in $I$ is denoted by $\cF(I)$. The $\alpha$-mixing coefficient was introduced by \cite{rosenblatt1956central}. \cite{doukhan1991mixing} defines this coefficient for random fields as
\begin{align*}
	\alpha(k) \coloneqq \sup_{ \substack{I, J \subseteq \Z^N ,\\ \dzn{I,J}\ge k}} \sup_{ \substack{ A \in \cF(I),\\ B \in \cF(J) } } \left| \p(A\cap B)-\p(A)\p(B)		\right|, \quad k\in\N.
\end{align*}
A random field is strongly spatial mixing if $\alpha(k)\rightarrow 0$ as $k\rightarrow\infty$. The $\beta$-mixing coefficient was introduced by \cite{kolmogorov1960strong}, it is defined for two sub-$\sigma$-algebras $\cF,\cG$ of $\cA$ as
\begin{align*}
		\beta(\cF,\cG) \coloneqq &\frac{1}{2} \sup \Bigg\{ \sum_{i\in I} \sum_{j\in J} \big|\p(U_i\cap V_j) - \p(U_i)\p(V_j) \big| \,:\\
		&\qquad\qquad\qquad (U_i)_{i\in I}\subseteq \cF, (V_j)_{j\in J} \subseteq \cG \text{ are finite partitions of } \Omega \Bigg\}.
\end{align*}
\cite{doukhan1991mixing} defines the $\beta$-mixing coefficient of the random field $Z$ as
$$
		\beta(k) \coloneqq \sup_{ \substack{I, J \subseteq \Z^N ,\\ \dzn{I,J}\ge k}} \beta( \cF(I),\cF(J)).
$$
The two mixing coefficients feature the relation that $2\alpha(k)\le\beta(k) \le 1$, see \cite{bradley2005basicMixing}. The definition of mixing coefficients for random fields differs from that of time series. The latter definition does not allow to consider interlaced subsets as the definitions of $\alpha(k)$ and $\beta(k)$ do. See also \cite{doukhan1991mixing} for a further discussion and the different properties of mixing time series.

In the following we associate to each $n\in\N_+^N$ an index sets $I_n \coloneqq \{s\in \Z^N: e_N\le s\le n \}$. As we study regression estimates based on data defined on this index set, we need to make precise the asymptotics of the vector $n$. Consider a sequence $(n(k): k\in\N)\subseteq \N^N$ such that
\begin{align}\label{Cond_IndexSet1}
\begin{split}
		 &\min\{ n_i(k):i=1,\ldots,N\} \ge C' \max\{ n_i(k):i=1,\ldots,N\}  \text{ for some } C'\in\R_+\\
		&\qquad\qquad\qquad \text{ and } \max\{ n_i(k+1):i=1,\ldots,N \}  > \max \{ n_i(k):i=1,\ldots,N \}.
\end{split}\end{align}
We say that such a sequence diverges to infinity in each component and write $n\rightarrow \infty$. Moreover, if $(A_{n(k)}:k\in\N)$ is a sequence which is indexed by the sequence $(n(k):k\in\N)$, we also write $A_n$ for this sequence. In particular, we characterize limits for real-valued sequences $A_n$ in this notation, i.e., we agree to write $\lim_{n\rightarrow\infty} A_n$ for $\lim_{k\rightarrow \infty} A_{n(k)}$.

\eqref{Cond_IndexSet1} allows us to proceed at different speeds in each direction, as long as the ratio between the minimum and the maximum does not fall below a certain level. The amendment that the running maximum is strictly increasing ensures that we select sufficiently many data points in the sampling process and guarantees a strongly universally consistent estimator.

We need two regularity conditions to prove the consistency of the sieve estimator. The first condition concerns both the index set on which the data is defined and the distribution of the data. We consider two models ($\alpha$) and ($\beta$):

\begin{condition}\label{regCond0}
$Z = \{Z(s): s\in \Z^N \}$ is an $\R^d$-valued random field for $N, d \in \N_+$ which has equal marginal distributions, i.e., $\cL_{Z(s)	} = \cL_{Z(t) } $ for all $s,t\in \Z^N$. Furthermore,

\begin{enumerate}[label=(\greek*)]

\item the $\alpha$-mixing coefficients decrease exponentially, i.e., $\alpha(k) \le c_0 \exp( -c_1 \, k)$, $k\in\N$ and for certain $c_0, c_1 \in \R_+$. \label{Cond_AlphaMixing}

\item for each pair $(s,t) \in \Z^N\times \Z^N$ the joint distribution of $(Z(s),Z(t))$ is absolutely continuous w.r.t. the product measure $\p_{Z(e_N)}\otimes \p_{Z(e_N)}$ such that the corresponding Radon-Nikod{\'y}m derivatives are uniformly bounded in that
\begin{align}\label{BoundRadonNikodymDerivative}
		\sup_{s,t\in\Z^N} \norm{ \frac{\intd{\p_{(Z(s),Z(t))}}}{\mathrm{d}\left(\p_{Z(e_N)} \otimes \p_{Z(e_N) }  \right)  } }_{\p,\infty} < \infty.
\end{align}
Moreover, the $\beta$-mixing coefficients of $Z$ decrease exponentially, i.e., $\beta(k) \le c_0 \exp( -c_1 \, k)$, $k\in\N$ and for certain $c_0, c_1 \in \R_+$.  \label{Cond_BetaMixing}
\end{enumerate}
\end{condition}

Condition~\ref{regCond0} \ref{Cond_AlphaMixing} is a very weak condition if the regression estimator is expected to be consistent. A usual assumption in this context is stationarity, as in \cite{hallin2004local} or \cite{li2016nonparametric}. However, since we want to cover irregular networks, we need this relaxed assumption because there is no definition of stationarity for random fields on a general (finite) network. Clearly, the dependence within the data has to vanish with increasing distance on the lattice. The decay of the $\alpha$-mixing coefficients is not unusual. One can show that for time series exponentially decreasing $\alpha$-mixing coefficients are guaranteed under mild conditions (\cite{davydov1973mixing}, \cite{Withers1981}).

Condition~\ref{regCond0} \ref{Cond_BetaMixing} implies the first condition. The assumption on the Radon-Nikod{\'y}m derivatives is reasonable as the dependence between the observations vanishes with increasing distance. Note that this condition does not imply that the marginal laws of the $Z(s)$ have to admit a density w.r.t.\ the Lebesgue measure. Assuming this condition, we obtain optimal rates of convergence. The technical reason why $\beta$-mixing ensures optimal rates is that the data $(X(s),Y(s))$ can be coupled with a sample $(X^*(s),Y^*(s))$ which is sufficiently independent, we give more details below. \cite{chen2013optimal} also obtain optimal rates of convergence for regression estimates of time series under a $\beta$-mixing condition.

At this point, it is important to remember the relation of Condition~\ref{regCond0} \ref{Cond_BetaMixing} to $m$-dependence. \cite{bradley1989caution} shows that for random fields $\beta$-mixing in the sense of the above definition and stationarity imply $m$-dependence, see also \cite{doukhan1991mixing}. In the following, we will work only in a single case with a combination of Condition~\ref{regCond0} \ref{Cond_BetaMixing} and stationarity. Especially, we do not need the assumption of stationarity when establishing rates of convergence under $\beta$-mixing. So our results apply indeed to $\beta$-mixing data and not exclusively to $m$-dependent data.

In order to study sieve estimators, we have to quantify the approximability of function classes by a finite collection of functions. For that reason, let $\epsilon>0$ and let $\left( \R^d,\cB(\R^d) \right)$ be endowed with a probability measure $\nu$. Consider a class of real-valued Borel functions $\cG$ on $\R^d$. Every finite collection $g_1,\ldots,g_M$ of Borel functions on $\R^d$ is called an $\epsilon$-cover of size $M$ of $\cG$ w.r.t.\ the $L^p$-norm $\norm{\,\cdot\,}_{L^p(\nu)}$ if for each $g\in\cG$ there is a $j$, $1\le j\le M$, such that $\norm{g-g_j}_{L^p(\nu)}^p = \int_{\R^d} |g-g_j|^p\intd{\nu}  < \epsilon$. The $\epsilon$-covering number of $\cG$ w.r.t.\ $\norm{\,\cdot\,}_{L^p(\nu)}$ is defined as
\begin{align}\label{DefConveringNumber}
		\mathsf{N}\left( \epsilon, \cG, \norm{\,\cdot\,}_{L^p(\nu)} \right) \coloneqq \inf\left\{ M \in \N: \exists\, \epsilon\text{-cover of }\cG\text{ w.r.t.\ }\norm{\,\cdot\,}_{L^p(\nu)} \text{ of size } M \right\}.
\end{align}
$\mathsf{N}$ is monotone, i.e., $\mathsf{N}\left(\epsilon_2, \cG, \norm{\,\cdot\,}_{L^p(\nu)} \right) \le \mathsf{N}\left(\epsilon_1, \cG, \norm{\,\cdot\,}_{L^p(\nu)}\right)$ if $\epsilon_1 \le \epsilon_2$. The covering number can be bounded uniformly over all probability measures under mild regularity conditions, compare the theorem of \cite{haussler1992decision} which is stated as Proposition~\ref{boundCoveringNumber} in the appendix. Since this last proposition involves the technical definition of the Vapnik-Chervonenkis-dimension, we mostly work in the following with a simple covering condition which is satisfied for any class $\cG$ of uniformly bounded functions.

\begin{condition}\label{coveringCondition}
$\cG$ is a class of uniformly bounded, measurable functions $f:\R^d \rightarrow \R$, i.e., there is a $B\in\R_+$ such that $\norm{f}_{\infty} = \sup\{|f(x)|:x\in\R^d\} \le B$ for all $f\in\cG$. Additionally, for all $\epsilon > 0$ and all $M\in\N_+$:
\par
\begingroup
\leftskip=1em
\rightskip=1em
\noindent
for any choice $z_1,\ldots,z_M \in \R^d$ the $\epsilon$-covering number of $\cG$ w.r.t.\ the $L^1$-norm of the discrete measure with point masses $M^{-1}$ in $z_1,\ldots,z_M$ is bounded above by a deterministic function $H_{\cG} (\epsilon)$ depending only on $\epsilon$ and $\cG$, i.e., $\mathsf{N}\left(\epsilon,\cG, \norm{\,\cdot\,}_{L^1(\nu)} \right) \le H_{\cG} (\epsilon)$, where $\nu = M^{-1} \sum_{k=1}^M \delta_{z_k}$.
\par
\endgroup
\end{condition}

The key requirement of Condition~\ref{coveringCondition} is that the covering number (which can be stochastic) admits a deterministic upper bound which only depends on the function class itself and on the parameter $\epsilon$. In particular, Condition~\ref{coveringCondition} is valid for classes of uniformly bounded Sobolev functions or H{\"o}lder continuous functions.

\subsection{The Estimation Procedure}
We assume that the random field $(X,Y)$ satisfies Condition~\ref{regCond0} \ref{Cond_AlphaMixing} or \ref{Cond_BetaMixing}. The $Y(s)$ are real-valued and each pair $(X(s),Y(s))$ satisfies the relation \eqref{lsqI} for each observation location $s\in \Z^N$. The error terms $\epsilon(s)$ are independent of $X$, have mean zero and unit variance. Note that we do not require any specific distribution of the error terms, e.g., a Gaussian distribution is not necessary. 

As in \cite{gyorfi}, let $\cF_{n(k)} \subseteq L^2(\mu_X)$ be deterministic increasing function classes the union of which is dense in $L^2\left(\mu_X\right)$. The function classes are indexed by the sequence from \eqref{Cond_IndexSet1}. In this context, increasing means that $\cF_{n(k)} \subseteq \cF_{n(k+1)}$ for $k\in\N_+$. Preferably, one chooses function classes which satisfy the universal approximation property, i.e., the union of the function spaces $\cF_n$ should be dense in $L^2(\mu)$ for any Borel probability measure $\mu$ on $\R^d$, see also \cite{hornik1991approximation}. This is quite useful because the distribution $\mu_X$ is unknown in practice. We come back to this in more detail in Section~\ref{Section_NonParaRegWavelets}.

The least-squares estimator is defined for a function class $\cF_n$ and a sample $\{(X(s),Y(s)): s\in I_{n}\}$ as
\begin{align}
		m_n \coloneqq \argmin{f\in\cF_n}  |I_{n}|^{-1} \sum_{s\in I_{n}} \big( Y(s) - f(X(s)) \big)^2. \label{lsqII}
\end{align}
Later, we will choose finite-dimensional linear spaces as $\cF_n$, i.e.,
\begin{align}
		\cF_n = \Bigg\{	\sum_{j=1}^{K_n} a_j f_j: f_j\colon\R^d\rightarrow\R, a_j \in \R, j=1,\ldots,K_n		\Bigg\}. \label{linSpaceAsFunctionClass}
\end{align}
Note that the basis functions $f_j$ are ordered, so that $\cF_{n(k)}\subseteq\cF_{n(k+1)}$ if $K_{n(k)} \le K_{n(k+1)}$. Using linear spaces as function classes has the computational advantage that the minimization is an unrestricted ordinary least-squares problem on the domain of the parameters without an additional penalizing term, i.e., the minimizing function in \eqref{lsqII} can be determined by the parameters $(a_1,\ldots,a_{K_n})$ which minimize
$
		|I_{n}|^{-1}  \sum_{s\in I_{n}}  ( Y(s) - \sum_{i=1}^{K_n} a_i f_i( X(s) ) )^2.
$

In examples of application, the parameters $a_i$ are estimated with principal component regression and singular value decomposition. Nevertheless, the subsequent results are derived for more general function classes $\cF_n$. These merely have to satisfy a technical condition on the measurability of the random variables $(X(s)$, $Y(s))$ mapping from the probability space $\pspace$ to $\R^d\times\R$; we indicate this by the writing $\Omega\ni \omega\mapsto (X(s,\omega),Y(s,\omega))$.

In what follows, let $(\rho_{n(k)}: k\in\N_+)$ be a real-valued and positive sequence which tends to infinity. Let $L>0$ and denote the truncation operator by $T_L y \coloneqq \max( \min(y,L), -L)$. Then we define the truncated function classes of $\cF_n$ by $T_{\rho_n} \cF_n \coloneqq \{T_{\rho_n}f : f\in\cF_n \}$. The function classes $\cF_n$, resp. $T_{\rho_n}\cF_n$, must not be too complex in the sense that taking the supremum preserves measurability: to be more precise, we need that the map
\begin{align}\begin{split}\label{supMeas2}
		\Omega \ni \omega \mapsto \sup_{f\in T_{\rho_n} \cF_n} &\bigg| 	 |I_{n} |^{-1 } \sum_{s\in I_{n} } \left|f(X(s,\omega)) - T_L Y(s,\omega)	\right|^2 \\
		& - \E{ \left|f(X(e_N) )- T_LY(e_N)	\right|^2}	\bigg| 
		\end{split}
\end{align}
is $\cA$-$\cB(\R)$-measurable for all $n(k)$ and for all $L>0$. This is necessary to apply exponential inequalities to \eqref{supMeas2}. Finite-dimensional linear spaces satisfy \eqref{supMeas2}, so this condition is satisfied in our applications. In order to obtain a consistent estimator in regions of $\R^d$ with sparse data, we consider the truncated least-squares estimator 
\begin{align}
		\hat{m}_n \coloneqq T_{\rho_n} m_n. \label{lsqIII}
\end{align}

Summing up, the properties of \eqref{lsqIII} are determined by the sample $\{(X(s),Y(s))\colon s\in I_n\}$, by the sequence $\rho_n$ and by the function classes $\cF_n$. In the case of linear spaces the latter are defined in terms of the number of basis functions $K_n$.

\subsection{Consistency and Rate of Convergence}
This subsection contains the main results of Section~\ref{Section_NonParaRegGeneral}. We start with a result on the consistency of the truncated least-squares estimator $\hat{m}_n$ from \eqref{lsqIII}.

\begin{theorem}\label{ConsistencyTruncatedLeastSquaresGeneral}
Let the random field $(X,Y)$ satisfy \eqref{lsqI}. Let the $\cF_n$ be increasing function classes the union of which is dense in $L^2(\mu_X)$ and which fulfil \eqref{supMeas2}. Assume that Condition \ref{coveringCondition} is satisfied for the truncated function classes $T_{\rho_n} \cF_n$ and define
$
	\kappa_{n} (\epsilon,\rho_n) \coloneqq \log H_{T_{\rho_n} \cF_n} \left(\epsilon/(4\rho_n) \right).
$
Assume that $\kappa_n (\epsilon,\rho_n) \rightarrow \infty$ as $n\rightarrow \infty$ in $\N^N$ for each $\epsilon>0$. 

Let Condition \ref{regCond0} \ref{Cond_AlphaMixing} be satisfied. If for each $\epsilon>0$
\begin{align}\label{EqConsistencyTruncatedLeastSquaresGeneral0}
	\kappa_n(\epsilon, \rho_n) \rho_n^4   (\log |I_n|)^2 / |I_n|^{1/N} \rightarrow 0 \text{ as } n \rightarrow \infty,
\end{align}
then the estimate $\hat{m}_n$ is weakly universally consistent, i.e.,
$$
	\lim_{n \rightarrow \infty} \E{ \int_{\R^d} | \hat{m}_n - m |^2 \intd{\mu_X} } = 0.
$$
Moreover, if additionally $(X,Y)$ is stationary and if additionally
\[
		\rho_n^4  \left( \log |I_n| \right)^4   / |I_n|^{1/N }  \rightarrow 0 \text{ as } n \rightarrow \infty,
\]
then $\hat{m}_n$ is strongly universally consistent, i.e., $\lim_{n \rightarrow \infty} \int_{\R^d} | \hat{m}_n - m |^2 \intd{\mu_X}  = 0$ $a.s.$

Let Condition \ref{regCond0} \ref{Cond_BetaMixing} be satisfied. If for each $\epsilon>0$
\begin{align*}
	\kappa_n(\epsilon, \rho_n) \rho_n^4  (\log |I_n|)^N / |I_n| \rightarrow 0 \text{ as } n\rightarrow \infty,
\end{align*}
then the estimate $\hat{m}_n$ is weakly universally consistent. Moreover, if $(X,Y)$ is stationary (and thus $m$-dependent) and if additionally
\[
	  \rho_n^4 (\log |I_n|)^{N+2} / |I_n|\rightarrow 0 \text{ as } n \rightarrow \infty,
\]
then $\hat{m}_n$ is strongly universally consistent.
\end{theorem}

The growth rates of the truncation sequence and of the covering number are upper bounds which guarantee a consistent estimator. We see that the conditions in the case of $\alpha$-mixing data are more restrictive than in the case of $\beta$-mixing data. In the first case the growth in $\rho_n^4$ times the logarithm of the covering number, $\kappa_n(\epsilon,\rho_n)$, have to be overcompensated by the $N$-th root of the sample size, $|I_n|^{1/N}$, for a weakly universally consistent estimator (modulo a logarithmic factor). In the second case of $\beta$-mixing data, the sample size $|I_n|$ is not corrected by the exponent $1/N$. This last result corresponds to the classical case of i.i.d.\ data, see \cite{gyorfi}. We will see this analogy between i.i.d.\ and $\beta$-mixing data again below. Moreover in both dependence settings, we need an additional growth condition which ensures a strongly universally consistent estimate. Note that in the case where Condition \ref{regCond0} \ref{Cond_BetaMixing} is satisfied and $(X,Y)$ is stationary, $(X,Y)$ is indeed $m$-dependent with the result of \cite{bradley1989caution}. Hence, the last statement in Theorem~\ref{ConsistencyTruncatedLeastSquaresGeneral} is actually achieved under $m$-dependence.

In the next corollary, we give an application to the linear spaces from \eqref{linSpaceAsFunctionClass}. In this case, we can compute an upper bound for the covering number with Proposition~\ref{boundCoveringNumber}. This corollary is also a generalized result of \cite{gyorfi} Theorem 10.3.

\begin{corollary}\label{ConsistencyTruncatedLeastSquaresGeneral_Corollary}
Let $\cF_n$ be the linear span of continuous and linearly independent functions $f_1,\ldots,f_{K_n}$ as in \eqref{linSpaceAsFunctionClass} such that $\cup_{k\in\N_+} \cF_{n(k)}$ is dense in $L^2(\mu_X)$. Assume Condition~\ref{regCond0} \ref{Cond_AlphaMixing}. $\hat{m}_n$ is weakly universally consistent if
$
	\lim_{n\rightarrow\infty} K_n \, \rho_n^4 \, \log \rho_n \,(\log |I_n|)^2 \,/  |I_n|^{1/N} = 0$. $\hat{m}_n$ is strongly universally consistent if additionally $(X,Y)$ is stationary and if additionally $
	\lim_{n\rightarrow\infty}\rho_n^4 \, (\log |I_n|)^{4} \,/ |I_n|^{1/N}  =0$.

Assume Condition~\ref{regCond0} \ref{Cond_BetaMixing}. If
$	\lim_{n\rightarrow\infty}K_n \, \rho_n^4 \, \log \rho_n \,(\log |I_n|)^N \,/  |I_n| = 0$, the estimate $\hat{m}_n$ is weakly universally consistent. $\hat{m}_n$ is strongly universally consistent if additionally $(X,Y)$ is stationary and if additionally $
	\lim_{n\rightarrow\infty}\rho_n^4 \, (\log |I_n|)^{N+2} \,/ |I_n| =0 
$.
\end{corollary}

One usually chooses a truncation sequence $\rho_n$ growing at a rate of $\cO(\log |I_n|)$ which is negligible, e.g., see \cite{kohler2003nonlinear} who considers piecewise polynomials as basis functions in the case of i.i.d.\ data.

The next result gives the rate of convergence of the truncated least-squares estimator $\hat{m}_n$ in both dependence scenarios. This rate can be divided into an empirical error which depends on the realization $\omega \in \Omega$ and an approximation error which relates the regression function $m$ to its projection onto the function classes $\cF_n$.

However, in order to derive a rate of convergence result, we need an additional requirement on the error terms because we do not rule out conditional dependence between two distinct observations $Y(s)$ and $Y(t)$. Thus, we need a condition on the conditional covariance matrix of the observations $Y(s)$ given the observations $X(s)$. We denote this matrix by $\operatorname{Cov}( Y(I_{n}) \,|\, X(I_{n}) ) $. Note that in the special case with uncorrelated error terms $\epsilon(s)$, $\operatorname{Cov}( Y(I_{n}) \,|\, X(I_{n}) )$ is a diagonal matrix and it is sufficient to impose a restriction on the conditional variances.

\begin{theorem} \label{nonParaRegRateOfConvergence}
Assume that the regression function and the conditional variance function are essentially bounded, i.e., $\norm{m}_{\infty}, \norm{ \varsigma}_{\infty}\le L$. If the error terms $\epsilon(s)$ are correlated, assume that $\E{ |\epsilon(e_N) |^{2+\gamma} } < \infty$ for some $\gamma >0$. The function classes $\cF_n$ are linear spaces as in \eqref{linSpaceAsFunctionClass}.

If Condition~\ref{regCond0} \ref{Cond_AlphaMixing} is satisfied, assume
$
		K_n \log(|I_n|) / |I_n|^{1/(2N)} \rightarrow 0 \text{ as } n \rightarrow \infty.
$
Then there is a $C\in\R_+$ such that
\begin{align*}
	 \E{ \int_{\R^d} \left| \hat{m}_n - m \right|^2 \intd{\mu_X} } \le 8 \inf_{f \in \cF_n} \int_{\R^d} \left|	f - m	\right|^2   \intd{\mu_X} + C \frac{ K_n \log |I_n|  }{ |I_n|^{1/(2N) } }.
\end{align*}

If Condition~\ref{regCond0} \ref{Cond_BetaMixing} is satisfied, assume
$
		K_n (\log |I_n| )^{N+2} / |I_n| \rightarrow 0 \text{ as } n \rightarrow \infty.
$
Then there is a $C\in\R_+$ such that
\begin{align*}
	 \E{ \int_{\R^d} \left| \hat{m}_n - m \right|^2 \intd{\mu_X} } \le 16 \inf_{f \in \cF_n} \int_{\R^d} \left|	f - m	\right|^2   \intd{\mu_X} + C \frac{ K_n (\log |I_n| )^{N+2} }{ |I_n| }.
\end{align*}
\end{theorem}

The first term appearing on the right-hand side of both inequalities is a multiple of the the approximation error which depends on the function class $\cF_n$ and the (unknown) function $m$. The second term is the estimation error. The number of basis functions $K_n$ has a linear influence on this error. This influence is negative because if $K_n$ increases, more parameters need to be estimated. Conversely, a growing sample size reduces this error. In the case of $\beta$-mixing data, an increasing sample size reduces the error more than in the case of $\alpha$-mixing data.

The boundedness of the functions $m$ and $\varsigma$ and the assumption that we know this bound are essential to derive rates of convergence, see \cite{gyorfi} for more details. However, note that we do not assume the error terms to be bounded. We only require a moment condition, i.e., $\E{ |\epsilon(e_N) |^{2+\gamma} } < \infty$ for some $\gamma>0$. This is not unexpected if we want to bound the summed conditional covariances in our model from \eqref{lsqI} which has the multiplicative heteroscedastic structure.

In the case of linear function spaces, \cite{gyorfi} find that the estimation error can be bounded by $K_k (\log k ) / k $ times a constant under similar assumptions for the case of an i.i.d.\ sample of size $k$. This guarantees an optimal rate of convergence in terms of \cite{stone1982optimal} up to a logarithmic factor. We see that for $\beta$-mixing data our result guarantees the same rate also up to a logarithmic factor. We discuss this in detail in Section~\ref{Section_NonParaRegWavelets} below.

\section{Linear Wavelet Regression with Spatially Dependent Data}\label{Section_NonParaRegWavelets}

In this section we consider an adaptive wavelet estimate of the regression function $m$.

\subsection{Preliminaries}
A detailed introduction to the properties of wavelets, in particular the construction of wavelets with compact support, can be found in \cite{meyer1995wavelets} and \cite{daubechies1992ten}. Since we consider $d$-dimensional data, we give a short review on important concepts of wavelets in $d$ dimensions  indexed by the lattice $\Z^d$. The definitions are taken from \cite{benedetto1993wavelets}. In the following, $M \in \R^{d\times d}$ is a matrix which preserves the lattice, i.e., $M \Z^d \subseteq \Z^d$. Moreover, $M$ is strictly expanding in that all eigenvalues $\zeta$ of $M$ satisfy $|\zeta| > 1$. Denote the absolute value of the determinant of $M$ by $|M|$.

A multiresolution analysis (MRA) of $L^2\left(\lambda^d \right)$ with a scaling function $\Phi:\R^d\rightarrow\R$ is an increasing sequence of subspaces $\ldots \subseteq U_{-1} \subseteq U_0 \subseteq U_1 \subseteq \ldots $ such that the following four conditions are satisfied
\begin{enumerate}[label=\textnormal{(\arabic*)}]
\item (Denseness) $\bigcup_{j \in \Z} U_j$ is dense in $L^2\left(\lambda^d \right)$,
\item (Separation) $\bigcap_{j \in \Z} U_j = \{0\}$,
\item (Scaling) $ f\in U_j $ if and only if $f( M^{-j} \,\cdot\,) \in U_0$,
\item (Orthonormality) $\{ \Phi(\,\cdot\, - \gamma): \gamma \in \Z^d \}$ is an orthonormal basis of $U_0$.
\end{enumerate}

The relationship between an MRA and an orthonormal basis of $L^2(\lambda^d)$ is summarized in the next theorem:

\begin{theorem}[Theorem 1.7 of \cite{benedetto1993wavelets}]\label{BenedettoTheorem}
Suppose $\Phi$ generates a multiresolution analysis and the $a_k(\gamma)$ satisfy for all $0\le j,k \le |M|-1$ and $\gamma\in \Z^d$ the equations
\begin{align*}
	\sum_{\gamma'\in \Z^d} a_j(\gamma')\, a_k(M\gamma + \gamma') = |M|\, \delta(j,k)\, \delta(\gamma,0) \quad\text{ and }\quad \sum_{\gamma\in\Z^d} a_0(\gamma) = |M|.
\end{align*}
Furthermore, define the functions $\Psi_k \coloneqq \sum_{\gamma\in\Z^d} a_k(\gamma)\, \Phi(M\,\cdot\, - \gamma)$ for $k=1,...,|M|-1$. Then the set of functions $\{ |M|^{j/2} \Psi_k( M^j\,\cdot\,-\gamma): j\in\Z, k=1,\ldots,|M|-1, \gamma\in\Z^d \}$ forms an orthonormal basis of $L^2(\lambda^d)$:
\begin{align}\begin{split}\label{ONBWavelets}
		&L^2(\lambda^d) = U_0 \oplus \bigoplus_{j\in\N} W_j = \bigoplus_{j\in\Z} W_j, \\
		&\qquad\text{ where } W_j \coloneqq \langle\, |M|^{j/2} \Psi_k(M^j\,\cdot\, - \gamma): k=1,\ldots,|M|-1,\gamma\in\Z^d \, \rangle.
\end{split}\end{align}
\end{theorem}

The scaling function $\Phi$ is also called the father wavelet and also denoted by $\Psi_0$. The $\Psi_k$ are the mother wavelets for $k=1,\ldots,|M|-1$. We sketch in a short example how to construct a $d$-dimensional MRA provided that one has a father and a mother wavelet on the real line.
\begin{example}[Isotropic $d$-dimensional MRA from one-dimensional MRA via tensor products]\label{MRAByTensorProduct}
Let $d\in\N_+$ and let $\phi$ be a father wavelet on the real line $\R$ together with a mother wavelet $\psi$, so that $\phi$ and $\psi$ are related by the identities
\begin{align*}
		\phi \equiv \sqrt{2} \sum_{\gamma\in\Z}  h_\gamma\, \phi( 2 \darg - \gamma) \text{ and } \psi \equiv \sqrt{2} \sum_{\gamma\in\Z} g_\gamma\, \phi (2 \darg -\gamma),
\end{align*}
for real-valued sequences $( h_\gamma: \gamma \in \Z)$ and $(g_\gamma: \gamma\in\Z)$. Let $\phi$ generate an MRA of $L^2(\lambda)$ with the corresponding spaces $U'_j$, $j\in \Z$. The $d$-dimensional wavelets are derived as follows. Define $M$ by $2 I_d$, where $I_d$ is the identity matrix in $\R^{d\times d}$. Denote the mother wavelets as pure tensors by $\Psi_k \coloneqq \xi_{k_1} \otimes \ldots \otimes \xi_{k_d}$ for $k \in \{0,1\}^d\setminus 0$, where $\xi_0 \coloneqq \phi$ and $\xi_1 \coloneqq \psi$. The scaling function is given as $\Phi \coloneqq \Psi_0 \coloneqq \otimes_{i=1}^{d} \phi$.

In Appendix~\ref{Appendix_Example} we demonstrate that $\Phi$ and the linear spaces $U_j \coloneqq \otimes_{i=1}^d U'_j$ form an MRA of $L^2(\lambda^d)$ and that the functions $\Psi_k$, generate an orthonormal basis in the sense of \eqref{ONBWavelets} where $W_j$ equals $\langle |M|^{j/2} \Psi_k\left( M^j\darg - \gamma \right): \gamma \in \Z^d,\, k \in \{0,1\}^d \setminus 0 \rangle$.
\end{example}

\subsection{Consistency and Rate of Convergence}
In the sequel, we bridge the gap between non-parametric regression and wavelet theory. As indicated in Theorem~\ref{ConsistencyTruncatedLeastSquaresGeneral} the function spaces $\bigcup_{k\in\N_+} \cF_{n(k)}$ are preferably dense in $L^2(\mu)$ for any probability measure $\mu$. The next theorem states that wavelets satisfy this universal approximation property.

\begin{theorem}\label{WaveletsDenseLp}
Consider an isotropic MRA on $\R^d$ with corresponding scaling function $\Phi$ constructed as in Example~\ref{MRAByTensorProduct} from a compactly supported real scaling function $\phi$. Let $\mu$ be a probability measure on $\cB(\R^d)$ and let $1\le p < \infty$. Then	$\bigcup_{j\in\Z} U_j \text{ is dense in } L^p(\mu)$.
\end{theorem}

In what follows, we assume that $\Phi$ is a compactly supported scaling function and that $M$ is a diagonalizable matrix, i.e., $M = S^{-1} D S$ for a diagonal matrix $D$ which contains the eigenvalues of $M$. Denote the maximum of the absolute values of the eigenvalues by $\zeta_{max} \coloneqq \max\{ |\zeta_i| : i=1,\ldots, d \}$. Set 
$$
\Phi_{j,\gamma} \coloneqq |M|^{j/2}\,\Phi ( M^j \,\cdot\, - \gamma ), \text{ where }\gamma\in\Z^d \text{ and } j\in\Z.
$$
Let $(w_{n(k)}: k\in\N) \subseteq \Z$ and $( j(n(k)):k\in\N) \subseteq \Z$ be two increasing sequences with $\lim_{n\rightarrow\infty} w_n= \infty $ and $\lim_{n\rightarrow\infty} j(n) = \infty$ such that $\lim_{n\rightarrow\infty} (\zeta_{max})^{j(n) }/{w_n}=0$. We set	$K_n \coloneqq \{\gamma\in\Z^d: \norm{\gamma}_{\infty} \le w_n \} \subseteq \Z^d$. Then, we define the linear function space by
\begin{align}
		\cF_n \coloneqq \bigg\{ \sum_{\gamma\in K_n} a_{\gamma}\, \Phi_{j(n),\gamma} : a_{\gamma}\, \in \R  \bigg\} \subseteq U_{j(n)}.  \label{LinSpaceWavelets}
\end{align}
So the $j(n)$ scale $\Phi$, whereas the $w_n$ control which translations are used for the construction of the function space $\cF_n$. Based on the results from the previous section, the following statements are true for the linear wavelet estimate.

\begin{theorem}\label{ConsistencyWaveletRegression}
Assume that the wavelet basis is dense in $L^2(\mu_X)$. Set $\rho_n \coloneqq c\, \log |I_n|$ for some constant $c\in\R_+$. Define the wavelet estimator $\hat{m}_n$ by \eqref{lsqIII} and \eqref{LinSpaceWavelets}.

Assume that Condition~\ref{regCond0}~\ref{Cond_AlphaMixing} is satisfied, then $\hat{m}_n$ is weakly universally consistent if 
\begin{align}\label{ConsistencyWaveletRegressionEq0}
		& \lim_{n \rightarrow \infty}  w_n^{d}\, (\log |I_n|)^{6}\, \log \log |I_n|\, \big/\, |I_n|^{1/N} = 0.
\end{align}
$\hat{m}_n$ is strongly universally consistent if $(X,Y)$ is stationary, if \eqref{ConsistencyWaveletRegressionEq0} holds and if additionally
$
		\lim_{n\rightarrow\infty} (\log |I_n|)^{8} \,\big/\,|I_n|^{1/N} = 0.
$

Assume that Condition~\ref{regCond0}~\ref{Cond_BetaMixing} is satisfied. If $\lim_{n \rightarrow \infty}  w_n^{d}\, (\log |I_n|)^{N+4}\, \log \log |I_n|\, \big/\, |I_n| = 0$, then $\hat{m}_n$ is weakly universally consistent. $\hat{m}_n$ is strongly universally consistent if additionally $(X,Y)$ is stationary and if additionally $\lim_{n\rightarrow\infty} (\log |I_n|)^{N+6} \,\big/\,|I_n| = 0$.
\end{theorem}

\begin{theorem}\label{RateOfConvergenceWaveletRegression}
Let Condition~\ref{regCond0}~\ref{Cond_BetaMixing} and the assumptions of Theorem~\ref{nonParaRegRateOfConvergence} be satisfied, then there is a constant $C$ independent of $n$ such that
\begin{align*}
\E{\int_{\R^d} |\hat{m}_n - m|^2\,\intd{\mu_X} } &\le C\,w_n^{d} (\log |I_n|)^{N+2} \big/ |I_n| + 16 \inf_{f\in\cF_n} \int_{\R^d} |f-m|^2\,\intd{\mu_X}.
\end{align*}
\end{theorem}

We give a short application in the case where the wavelet basis is generated by isotropic Haar wavelets in $d$ dimensions and where the regression function $m$ is $(A,r)$-H{\"o}lder continuous on a compact subset of $\R^d$. This means that $	|m(x)-m(y)| \le A \norm{x-y}_{\infty}^r $ for all $x,y$ in the domain of $m$, for an $A\in \R_+ $ and for an $r\in (0,1]$.

\begin{corollary}\label{HaarBasisRateOfConv}
Let the conditions of Theorem \ref{RateOfConvergenceWaveletRegression} be satisfied such that the data fulfils Condition~\ref{regCond0} \ref{Cond_BetaMixing}. Let the conditional mean function $m$ be $(A,r)$-H{\"o}lder continuous. Define the level $j$ as a function of $n$ by
$
				2^j \simeq |I_n|^{1/(d+2r)}$. Then
\begin{align}\label{ExampleRateOfConvHolder}
		\E{ \int_{\R^d} | \hat{m}_n - m|^2\, \intd{\mu} } & = \cO\left(	(\log |I_n|)^{N+2} |I_n|^{-2r/(d+2r) } \right).
\end{align}
\end{corollary}
\begin{proof}
Note that by construction it suffices to choose $w_n$ proportional to $2^j$ because the domain of the function $m$ is bounded, we can cover it with $2^{j d} $ wavelets from the $j$-th scale. This means that the estimation error behaves as $\cO( 2^{j d} (\log |I_n|)^{N+2} / |I_n| )$.

It remains to compute the approximation error: there is a function $f \in \cF_n$ piecewise constant on dyadic $d$-dimensional cubes of edge length $2^{- j }$ with values
\[
f(x) = m\left( (\gamma_1,\ldots,\gamma_d)/2^j \right) \text{ for } x\in \left[ (\gamma_1,\ldots,\gamma_d)/2^j, ((\gamma_1,\ldots,\gamma_d) + e_N)/2^j \right),
\]
where $\gamma_i \in \Z$ for $i=1,\ldots,d$ such that $\gamma=(\gamma_1,\ldots,\gamma_d)$ is an admissible element from $K_n$. Hence for this $f$
$$
			\int_{\R^d} |f-m|^2\,\intd{\mu_X} \le \sup_{ dom\, m} |f-m|^2 \le A^2\, 2^{-2r j(n) }.
$$ 
The choice of $j$ as $2^j \simeq |I_n|^{1/(d+2r)}$ approximately equates the estimation and the approximation error.
\end{proof}

The interpretation of the two parameters $d$ and $r$ in the rate of convergence is well known: on the one hand, an increase in $d$ deteriorates the rate (the curse of dimensionality). On the other hand, an increase in $r$ towards 1 increases the rate of convergence because the regression function becomes smoother and can be better approximated by finite linear combinations of functions.

We compare the above result to the results for the classical case of i.i.d.\ data: if the regression function is H{\"o}lder continuous, the rate of convergence is in $\cO\left( k^{-2r/(d+2r)} \right)$ up to a logarithmic factor, where the sample size is $k$, see \cite{kohler2003nonlinear} or \cite{gyorfi}. This is nearly optimal when compared to \cite{stone1982optimal}. The additional log-loss is due to the increasingly complex sieves.
 
Hence, our rate of $\cO\left(	(\log |I_n|)^{N+2} |I_n|^{-2r/(d+2r) } \right)$ is the same modulo a logarithmic factor. Note that this result is independent of the lattice dimension $N$ on which the data is defined.

\cite{chen2013optimal} also consider regression function estimates with wavelets for $\beta$-mixing time series. Their results can be compared to the present findings in the special case where the lattice dimension $N$ equals 1. They also obtain a nearly optimal rate w.r.t.\ the sup-norm.

\cite{li2016nonparametric} considers a wavelet based regression estimator for spatially dependent data similar to our model \eqref{lsqI} and also obtains a nearly optimal rate. However, some of the regularity conditions are more restrictive than those in Condition~\ref{regCond0}: the design distribution of the regressors $X(s)$ has to admit a known density and the response variables $Y(s)$ have to be bounded. Our results are derived without these additional restrictions.

\section{Examples of Application}\label{ExamplesOfApplication}

We begin this section with some well-known results on random fields necessary for the following applications. Let $G=(V,E)$ be a finite graph. We write $\text{Ne}(s)$ for the neighbours of a node $s$ w.r.t.\ the graph $G$ and $-s$ for the set $V\setminus \{s\}$.

Assume that $( Y(s)\colon s\in V )$ is multivariate normally distributed with expectation $\alpha \in \R^{|V|}$ and covariance matrix $\Sigma \in \R^{ |V|\times |V|}$.
If we write $P$ for the precision matrix $\Sigma^{-1}$, the conditional distribution of $Y(s)$ given the remaining observations $Y(-s)$ is
\[
		Y(s) \, |\, Y(-s) \sim \cN\Big(  \alpha(s) - (P(s,s))^{-1}  \sum_{t \neq s} P(s,t) \Big( y(t) - \alpha(t) \Big)  , P(s,s)^{-1} \Big).
\]
Since $P = \Sigma^{-1}$ is symmetric and since we can assume that $P(s,s) ^{-1} > 0$, $Y$ is a Markov random field if and only if
$
		P(s,t) \neq 0 \text{ for all } t \in \operatorname{Ne}(s) \text{ and }
		P(s,t) = 0 \text{ for all } t \in V \setminus \operatorname{Ne}(s),
$
 for all nodes $s\in V$.

\cite{cressie1993statistics} investigates the conditional specification
\begin{align}\label{concliquesMVN}
		Y(s) \,|\, Y(-s) \sim \cN\Big( \alpha(s) + \sum_{t \in \operatorname{Ne}(s)} c(s,t) \big(Y(t) - \alpha(t) \big), \tau^2(s) \Big),
\end{align}
where $C=\big( c(s,t) \big)_{1\le s,t\le |V|}$ is a $|V|\times |V|$ matrix and $T = \text{diag}(\tau^2(s): s\in V)$ is a $|V|\times|V|$ diagonal matrix such that the coefficients satisfy the condition $\tau^2(s) c(t,s) = \tau^2(t) c(s,t)$ for $s\neq t$ and $c(s,s) = 0$ as well as $c(s,t) = 0 = c(t,s)$ if $s,t$ are not neighbours. This means $P(s,t) = -c(s,t) P(s,s)$, i.e., $\Sigma^{-1} = P =  T^{-1}  (I-C)$. If $I-C$ is invertible and if $(I-C)^{-1} T$ is symmetric and positive definite, then the entire random field is multivariate normal with	$Y \sim \cN\left( \alpha, (I-C)^{-1} T \right)$.

It is plausible to use equal weights $c(s,t)$ in many applications, see \cite{cressie1993statistics}. Thus, we can write the matrix $C$ as $C=\eta H$, where $H$ is the adjacency matrix of $G$, i.e., $H(s,t)$ is 1 if $s$ and $t$ are neighbours, otherwise it is 0. Denote the maximal (resp. minimal) eigenvalue of $H$ by $h_m$ (resp. $h_0$). Assume that $h_0<0<h_m$ which is often satisfied in applications. We know from the properties of the Neumann series that in this case the matrix $I-C$ is invertible if $(h_0)^{-1} < \eta < (h_m)^{-1}$.

This insight allows us to simulate a Gaussian Markov random field with an MCMC-algorithm using concliques with a full conditional distribution. Here we refer to \cite{kaiser2012} for a general introduction to the concept of concliques and the simulation procedure, the latter is also described in \cite{krebs2017orthogonal}. In the present simulation examples, we run $15k$ iterations of the MCMC-algorithm. These suffice to ensure a nearly stationary distribution of the Gaussian random field.

We sketch the simulation procedure: let $V=\{s_1,\ldots,s_{|V|}\}$ be finite. We simulate a $d$-dimensional random field $Z$ on $G$ such that each component $Z_i$ takes values in $\R^{|V|}$, for $i=1,\ldots,d$. We use copulas to obtain a dependence-structure between the $Z_i$. Each $Z_i$ has a specification
$
	Z_i \sim  \cN \big( \alpha\, (1,\ldots,1)' , \sigma^2 \Sigma \big),
$
where $\alpha\in\R$, $\sigma\in\R_+$ may depend on $i$. Furthermore, $\Sigma$ is a correlation matrix which satisfies the relation
\begin{align}
		( I - \eta H )^{-1} T = \sigma^2 \Sigma. \label{GMRFSim2}
\end{align}
The parameter $\eta$ is chosen such that $I - \eta H$ is invertible and $T$ is a diagonal matrix $T = \text{diag}\big( \tau^2(s_1),\ldots,\tau^2(s_{|V|}) \big)$. A large absolute value of $\eta$ indicates a strong dependence within the random variables of one component $Z_i$, whereas $\eta = 0$ indicates independence within the component. The marginal distributions within the $i$-th component equal each other, i.e., $Z_i(s) \sim \cN(\alpha,\sigma^2)$ for $s\in V$. However, the conditional variances $\tau^2(\,\cdot\,)$ within a component $Z_i$ may differ.

In the next step, we use some of the components to construct the random field $\{ X(s): s \in V\}$ and use another independent component to construct the error terms $\{\epsilon(s): s\in V \}$. We specify this below. Then we simulate the random field $Y$ as in \eqref{lsqI} for a choice of $m$ and a constant $\varsigma$. So the conditional heteroscedastic part is constant. However, depending on the underlying graph $G$, the error terms $\epsilon(s)$ can have a complex mutual dependence pattern. We estimate $m$ with the truncated least-squares estimator from \eqref{lsqIII}. In the situation where the regression function $m$ is known, the $L^2$-error can serve as a criterion for the goodness-of-fit of $\hat{m}$: we partition the index set $V$ into a set $V_L$ containing the locations for the learning sample and a set $V_T$ containing the locations for the testing sample. Here both $V_L$ and $V_T$ should be two connected sets w.r.t.\ the underlying graph if this is possible. We estimate $\hat{m}$ from the learning sample and compute the approximate $L^2$-error with Monte Carlo integration over the testing sample, i.e.,
\begin{align}\label{GMRFSim3}
	\int_{\R^d} |\hat{m} - m|^2 \intd{\mu_X} \approx  |V_T|^{-1} \sum_{s\in V_T} | \hat{m}( X(s) ) - m( X(s) ) |^2.
\end{align}
We run this entire simulation procedure $1000$ times. Afterwards, we compute the mean and the standard deviation of the (approximate) $L^2$-error from \eqref{GMRFSim3} based on these simulations.

\begin{example}[Bivariate non-parametric regression]\label{SimGMVNEx}
We simulate a random field on a planar graph $G=(V,E)$ which represents the administrative divisions in the Sydney bay area on the statistical area level 1. See the website of the Australian bureau of statistics (www.abs.gov.au) for further reference. It comprises 7,713 nodes and approximately 47k edges in total. Hence, $G$ is highly connected relative to the four-nearest neighbour structure. Figure~\ref{fig:SydneyBayArea_SA1Graph} illustrates the graph.

We model a three-dimensional Markov random field $Z=(Z_1, Z_2, Z_3)$. Every $Z_i$ has a specification as in \eqref{concliquesMVN} such that the marginals $Z_i(s)$ within each component are standard normally distributed. The parameter space of $\eta$ is derived from the adjacency matrix of the graph $G$ and contains the interval $(-0.2221, 0.1312)$. Note that the range for the lattice with a four-nearest-neighbourhood structure is $\left( -0.25, 0.25 \right)$.

Then we adjust the marginal conditional variance $\tau^2_i(s)$ of the variable $Z_i(s)$ such that the entire random vector $Z_i$ has a covariance structure of the type $\Sigma_i$ as in \eqref{GMRFSim2}.
\begin{figure}
	\begin{subfigure}[b]{\textwidth} \centering
		\includegraphics[height=8cm, keepaspectratio]{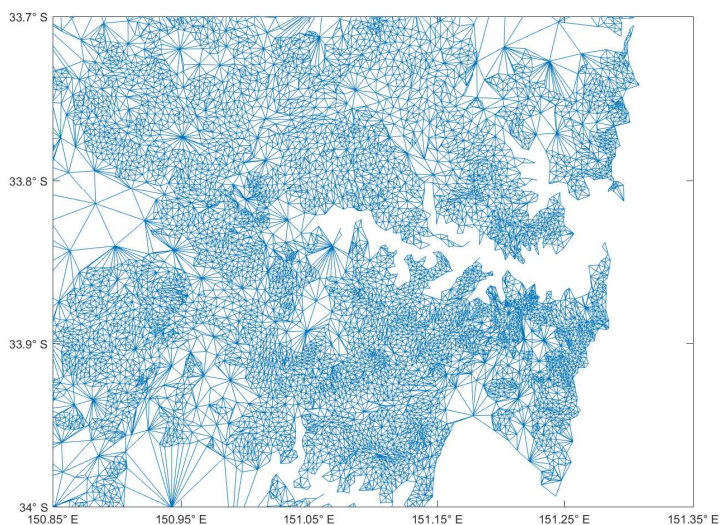}
	\caption{The Sydney bay area (on statistical area level 1-scale)}
	\label{fig:SydneyBayArea_SA1Graph}
	\end{subfigure}
	
	\begin{subfigure}[b]{\textwidth} \centering
			\includegraphics[height = 9cm, keepaspectratio]{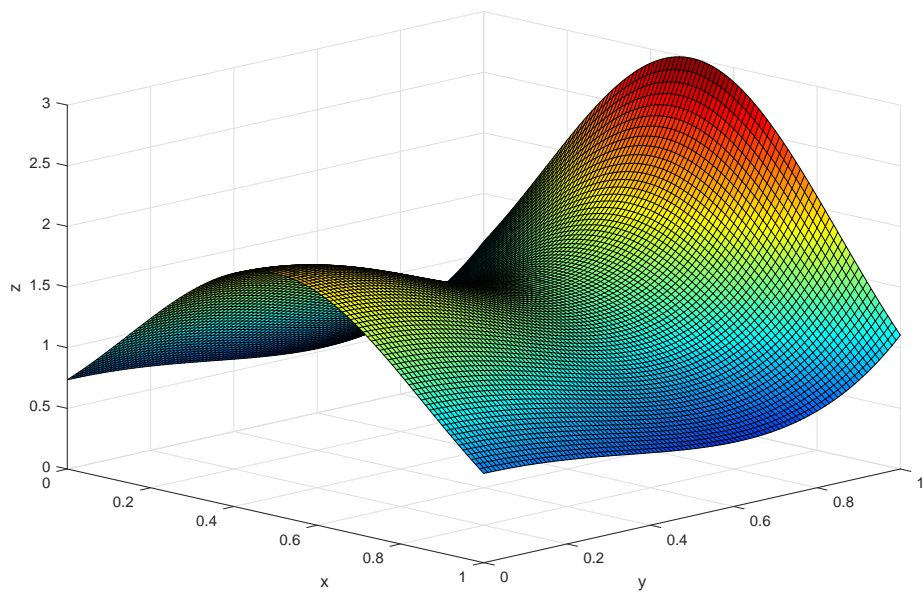}
	\caption{Function plot of $m$}
	\label{fig:Regression2dExample1_Theoretic}
	\end{subfigure}	
\caption{Input graph and regression function for the bivariate regression problem}
\label{fig:Regression2DExample1_Fig1}
\end{figure}

\begin{figure}
	\begin{subfigure}[b]{\textwidth} \centering
			\includegraphics[height = 8cm, keepaspectratio]{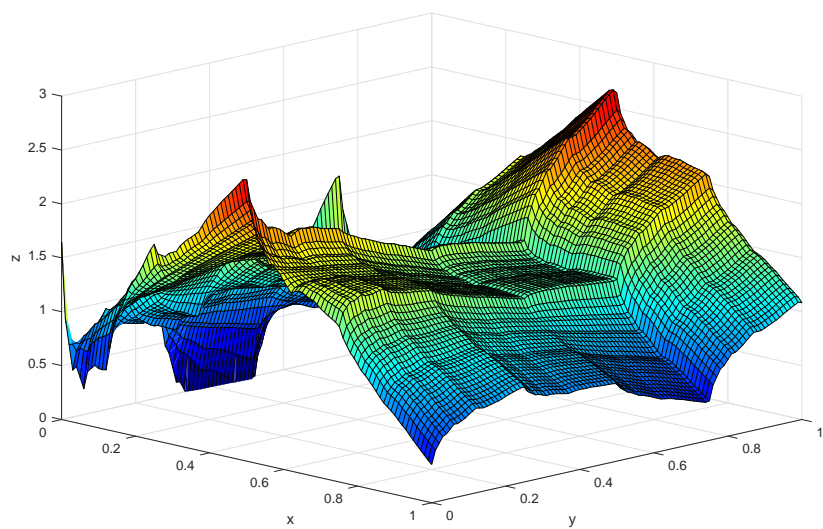}
		\caption{Estimate of ${m}$ with the D4 scaling function for $j=2$}
		\label{fig:Regression2dExample1_D2}
	\end{subfigure}

	\begin{subfigure}[b]{\textwidth} \centering
			\includegraphics[height = 8cm, keepaspectratio]{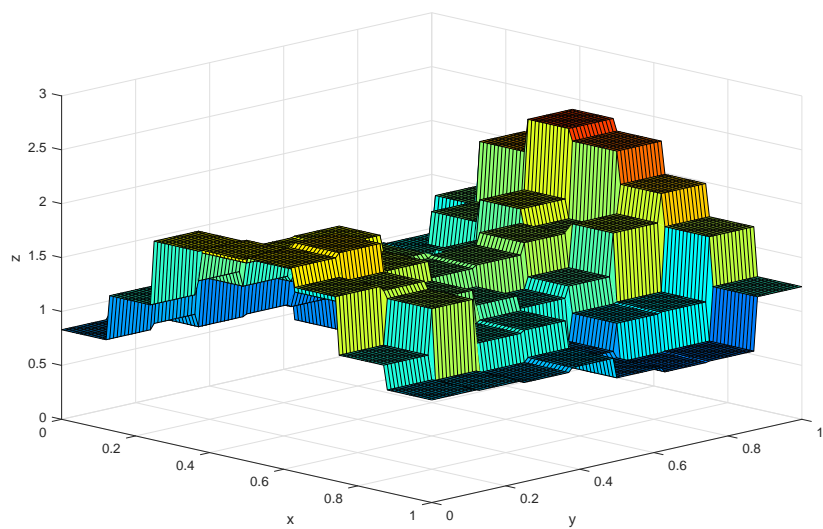}
		\caption{Estimate of ${m}$ with the Haar scaling function for $j=3$}
		\label{fig:Regression2dExample1_Haar}
	\end{subfigure}	
\caption{Estimated regression functions for a bivariate regression problem}
\label{fig:Regression2DExample1_Fig2}
\end{figure}

\begin{table}[ht]
\begin{center}
\begin{tabular}{|| c || c | c || c | c ||  }
\hline
\hline
	&	\multicolumn{2}{| c ||}{Estimates on the graph} & \multicolumn{2}{| c ||}{Independent reference estimates} \\
\hline
\hline	
	j &	D4 wavelet & Haar wavelet & D4 wavelet & Haar wavelet \\
\hline
\hline
\multirow{2}{*}{1}	&	0.264	&	0.413	&	0.260	&	0.406	\\
	&	(0.006)	&	(0.008)	&	(0.006)	&	(0.007)	\\
\hline									
\multirow{2}{*}{2}	&	0.122	&	0.258	&	0.119	&	0.254	\\
	&	(0.009)	&	(0.008)	&	(0.009)	&	(0.007)	\\
\hline									
\multirow{2}{*}{3}	&	0.163	&	0.198	&	0.170	&	0.196	\\
	&	(0.036)	&	(0.010)	&	(0.044)	&	(0.010)	\\
\hline									
\multirow{2}{*}{4}	&	0.422	&	0.259	&	0.435	&	0.257	\\
	&	(0.075)	&	(0.012)	&	(0.077)	&	(0.012)	\\
\hline
\hline
\end{tabular}
\caption{$L^2$-error of the bivariate regression problem: the estimated mean and in parentheses the estimated standard deviation for a level $j=1,\ldots,4$. The first two columns give the results for the random field, the last two columns those of the independent reference sample.}
\label{TableExampleBivariate}
\end{center}
\end{table}

In order to obtain dependent components $Z_1$ and $Z_2$, we draw the error terms from a two-dimensional Gaussian copula in each iteration. The exact simulation parameters are given by
		$\mu_{Z_i} = 0$, $\sigma_i = 1$ for $i = 1,2,3$, $\eta_1 = 0.12$, $\eta_2 = -0.18$ and $\eta_3 = 0.12$. The covariance between the first two components is 0.7. The third component $Z_3$ is simulated as independent. The vectors $\tau^2_i\in \R^{|V|}$ are computed with the formula $\tau_i^2(s) = \{	\operatorname{diag}^{\sim}(	\text{inv}(I-\eta_i H)\,	)	\}^{-1}(s)$ for $i=1,2,3$. Here we denote the inverse of a matrix by $\operatorname{inv}$, the operator that maps the diagonal of a matrix to a vector by $\operatorname{diag}^{\sim}$ and the elementwise inversion of a vector by $\{\cdot\}^{-1}$. Afterwards, we transform the first two components $Z_1$ and $Z_2$ with the distribution function of a two-dimensional standard normal distribution onto the unit square and obtain the random field $(X_1,X_2)$. We specify the mean function in this example as
\[
		m\colon \R^2 \rightarrow \R,\, (x_1,x_2)\mapsto (2 - 3x_2^2+ 4x_2^4 ) \, \exp\big(- ( 2 x_1 - 1)^2	 \big).
\]

Figure~\ref{fig:Regression2dExample1_Theoretic} shows the function plot of $m$. We simulate $Y(s) = m( X_1(s), X_2(s) ) + Z_3(s)$ and we use two different wavelet scaling functions for the estimation of $m$: we perform the first regression with the Haar scaling function $\phi = 1_{[0,1)}$ and the second with Daubechies 4-scaling function $D4$ (which is also known as $db2$). Figure~\ref{fig:Regression2DExample1_Fig2} displays the results: Figure~\ref{fig:Regression2dExample1_D2} depicts the estimate with Daubechies 4-scaling function, Figure~\ref{fig:Regression2dExample1_Haar} the one with the Haar scaling function. Table~\ref{TableExampleBivariate} gives the $L^2$-error statistics. Note that the $L^2$-error minimizing $j$ for the Haar wavelet differs from the $j$ minimizing the error for Daubechies 4-scaling function. Moreover, the $D4$ wavelet outperforms the Haar scaling function in this example. Table~\ref{TableExampleBivariate} also shows the $L^2$-error statistics for the same regression problem but with i.i.d.\ data of the same sample size. Note that the estimator obtained from i.i.d.\ data is slightly better than the estimate from the random field for both wavelet types.
\end{example}

\begin{example}[Univariate non-parametric regression]
In this example we consider a one-dimensional spatial regression problem based on a graph which represents Australia divided into administrative divisions on the statistical area level 3. The graph consists of 330 nodes and 1600 edges, cf. Figure~\ref{fig:AustraliaGraphSA3}. This graph is highly connected in certain regions relative to the four-nearest neighbour structure on a lattice.

\begin{figure}
			\begin{subfigure}[b]{\textwidth} \centering
	\centering
		\includegraphics[height = 8cm, keepaspectratio]{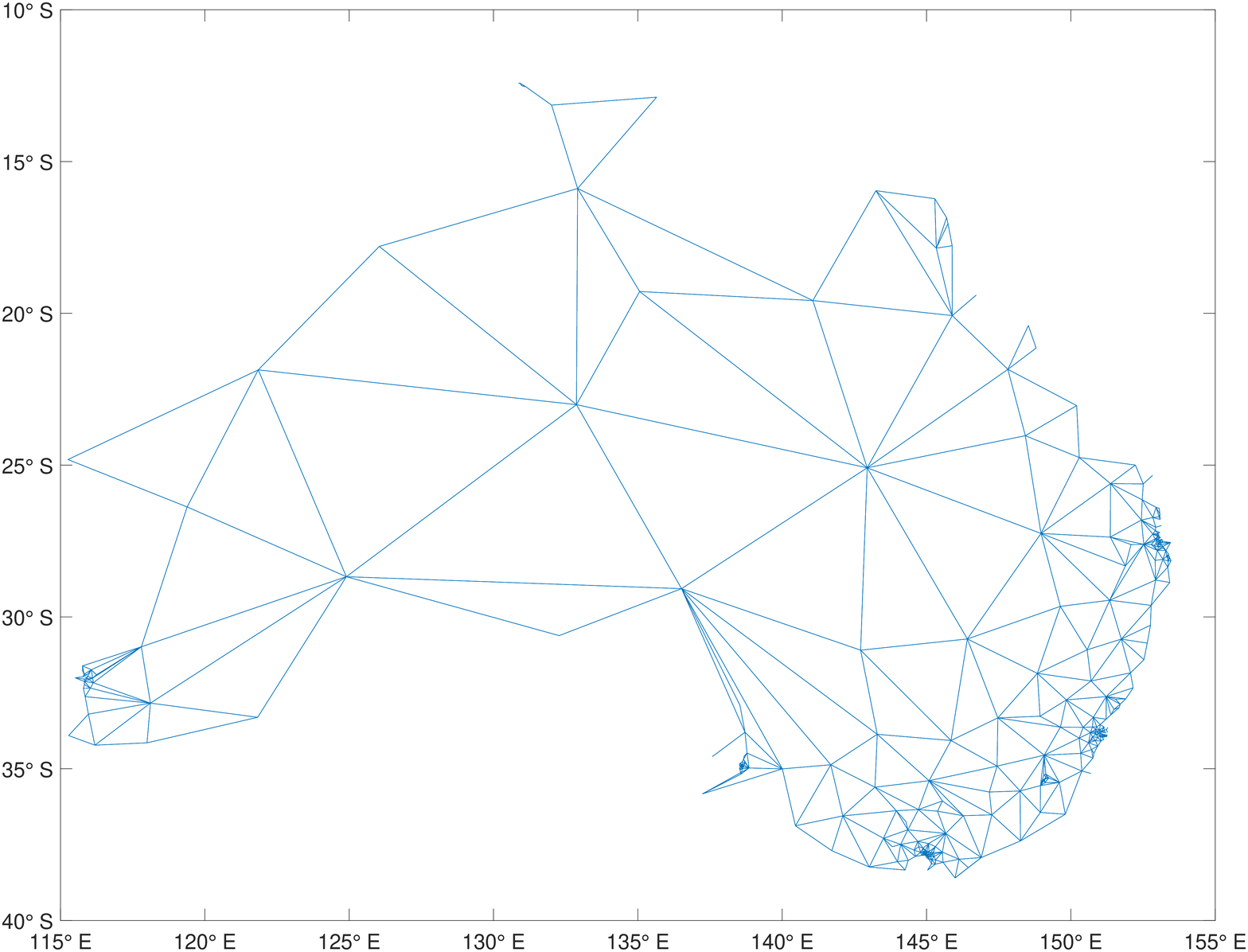}
	\caption{Administrative divisions of mainland Australia}
	\label{fig:AustraliaGraphSA3}
\end{subfigure}

	\begin{subfigure}[b]{\textwidth} \centering
		\includegraphics[height = 8cm, keepaspectratio, trim = 0 0 0 0, clip=true]{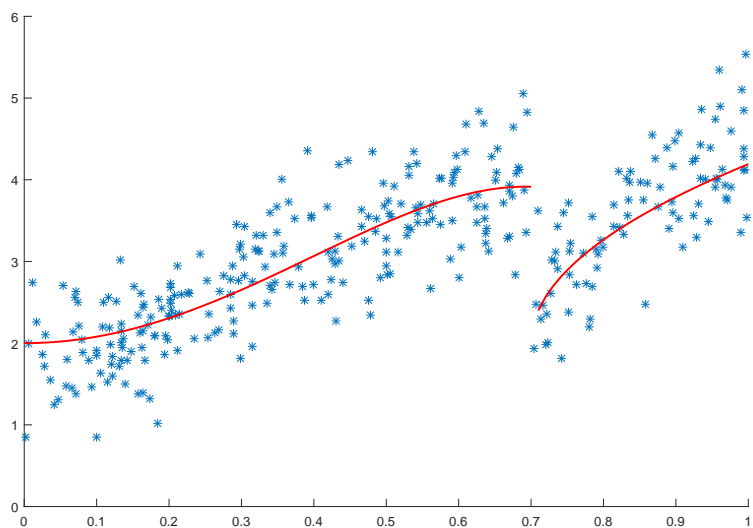}
	\caption{A realization of $X$ and the mean function $m$}
	\label{fig:Regression1dExample1_SimulatedvsTheoretic}
\end{subfigure}

\caption{Graph and true regression function.}
\label{fig:RegressionFunctionEstimates}
\end{figure}

\begin{figure}
	\begin{subfigure}[b]{\textwidth} \centering
		\includegraphics[height = 8cm, keepaspectratio, trim = 0 0 0 0, clip=true]{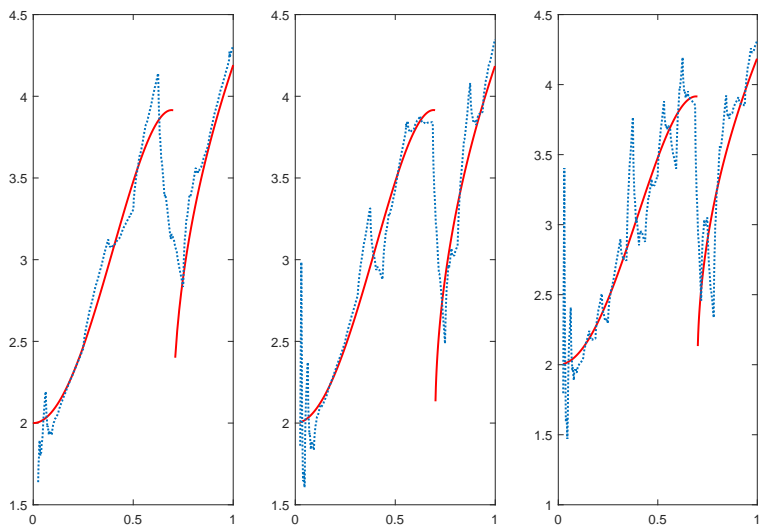}
	\caption{Estimate of ${m}$ with the D4 scaling function for $j=3,4,5$.}
	\label{fig:Regression1dExample1_D4}
\end{subfigure}

		\begin{subfigure}[b]{\textwidth} \centering
		\includegraphics[height = 8cm, keepaspectratio, trim = 0 0 0 0, clip=true]{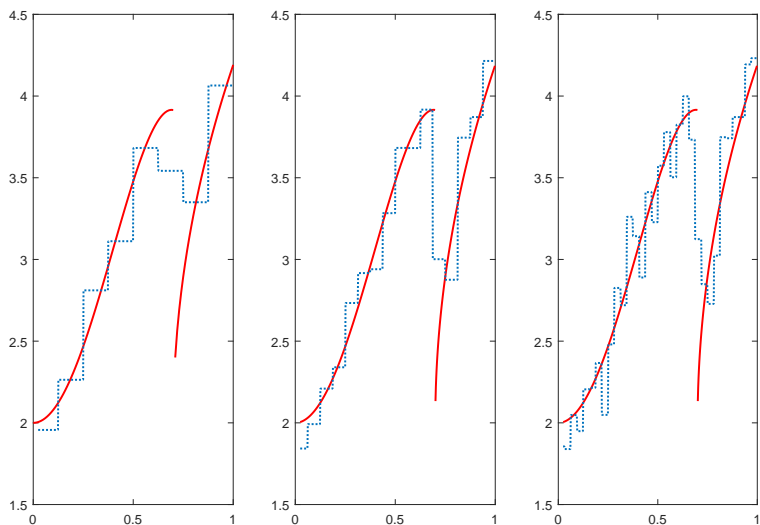}
	\caption{Estimate of ${m}$ with the Haar scaling function for  $j=3,4,5$.}
	\label{fig:Regression1dExample1_Haar}
\end{subfigure}
\caption{The estimates for the univariate regression problem.}
\label{fig:RegressionFunctionEstimates2}
\end{figure}

We simulate two Gaussian random fields $Z_1$ and $Z_2$ on $G$ with marginal means 0 and marginal variances 1 with the Markov chain method as in Example~\ref{SimGMVNEx}. The parameter space for $\eta$ contains the interval $(-0.3060, 0.1615)$. We choose $\eta$ for both components equal to 0.15 and run $15k$ iterations of the MCMC-algorithm. Then we use the inverse of the standard normal distribution to retransform the component $Z_1$ onto the unit interval and obtain the random field $X$ with marginals uniformly distributed on $[0,1]$. The conditional mean function is defined as the discontinuous function
\[
		m\colon [0,1] \rightarrow \R,\, x\mapsto \big(2 + 8x^2 - (1.7x)^4\big) 1_{\{ x \le 0.7 \}} + 2 \big(\sqrt{ 4 (x - 0.7)}+ 1\big) 1_{ \{0.7 < x \} }.
\]

We define $Y(s) = m( X(s) ) + Z_2(s) / 2$. Figure~\ref{fig:Regression1dExample1_SimulatedvsTheoretic} depicts the simulated random field. Figure~\ref{fig:Regression1dExample1_D4} shows the estimation with the Daubechies 4-scaling function, while Figure~\ref{fig:Regression1dExample1_Haar} depicts the result for the Haar wavelet. We infer from Table~\ref{TableExampleUnivariate} that the $L^2$-error is minimized for the level $j=4$ in all cases. Note that in this example the Daubechies wavelet always outperforms the Haar wavelet when measured by the $L^2$-error. Again, the $L^2$-error of the independent reference estimate is slightly better in each case.

\begin{table}[ht]
\begin{center}
\begin{tabular}{|| c || c | c || c | c ||  }
\hline
\hline
	&	\multicolumn{2}{| c ||}{Estimates on the graph} & \multicolumn{2}{| c ||}{Independent reference estimates} \\
\hline
\hline	
	j &	D4 wavelet & Haar wavelet & D4 wavelet & Haar wavelet \\
\hline
\hline
\multirow{2}{*}{2}	&	0.326	&	0.405	&	0.321	&	0.401	\\
	&	(0.031)	&	(0.059)	&	(0.029)	&	(0.061)	\\
\hline									
\multirow{2}{*}{3}	&	0.241	&	0.344	&	0.233	&	0.341	\\
	&	(0.033)	&	(0.064)	&	(0.035)	&	(0.067)	\\
\hline									
\multirow{2}{*}{4}	&	0.224	&	0.284	&	0.213	&	0.280	\\
	&	(0.077)	&	(0.073)	&	(0.062)	&	(0.078)	\\
\hline									
\multirow{2}{*}{5}	&	0.319	&	0.349	&	0.299	&	0.333	\\
	&	(0.172)	&	(0.117)	&	(0.134)	&	(0.093)	\\
\hline									
\multirow{2}{*}{6}	&	0.772	&	0.753	&	0.712	&	0.727	\\
	&	(0.437)	&	(0.213)	&	(0.380)	&	(0.212)	\\
\hline
\hline
\end{tabular}
\caption{$L^2$-error of the univariate regression problem: the estimated mean and in parentheses the estimated standard deviation for a level $j=2,\ldots,6$. The first two columns give the results for the random field, the last two columns those of an independent reference sample of the same size.}
\label{TableExampleUnivariate}
\end{center}
\end{table}
\end{example}

\section{Proofs of the Results in Section~\ref{Section_NonParaRegGeneral} and Section~\ref{Section_NonParaRegWavelets} }
\label{Section_Proofs}

The first lemma is a consequence of the coupling lemma of \cite{berbee1979random}. In the case of $\beta$-mixing, we construct another sample $(X^*,Y^*)$ which has good properties.

\begin{lemma}\label{CouplingBetaMixing}
Let $(X,Y)$ be a random field on $\Z^N$. For each $n\in\N_+^N$ and $q\in\N_+$ such that $2q < \min\{n_i: i=1,\ldots,N \}$, there is a partition of $I_n$ which is denoted by $\{I(l,u):l=1,\ldots,2^N, u=1,\ldots,R\}$ and collection of random variables $Z^*(l,u) = ((X^*(s),Y^*(s)):s\in I(l,u) ) \in \R^{(d+1)q^N}$ such that for each $l$ the collection $Z^*(l,1),\ldots,Z^*(l,R)$ is independent
and $\p( Z^*(l,u) \neq Z(l,u) ) = \beta(q)$ where $Z(l,u)=((X(s),Y(s)):s\in I(l,u) )$ and $\beta$ is the $\beta$-mixing coefficient of $(X,Y)$. Moreover, $Z^*(l,u)$ is independent of $Z(l,1),\ldots,Z(l,u-1)$ for each $u=1,\ldots,R$ and for each $l=1,\ldots,2^N$.
\end{lemma}
\begin{proof}
The proof follows as in \cite{carbon1997kernel} where a similar coupling result is established under the $\alpha$-mixing condition. We only sketch the main parts. Firstly, we give the construction of the partition. We choose $R_1, \ldots, R_N$ such that
\begin{equation*}
2q (R_i-1) < n_i \le 2q R_i =: n_i^* \text{ for each } k=1, \ldots, N.
\end{equation*}
For the $k$-th coordinate direction, we partition the summation index set
$\{1, \ldots, n_i^*\} \supseteq \{ 1, \ldots, n_i\}$ into $R_i$ subsets each
consisting of two disjoint intervals of length $q$. So, we
have a union of $2 R_i$ intervals of length $q$.

Combining the partitions in all $N$ coordinate directions, we get a partition of the
$N$-dimensional rectangle $I_{n^*} = \{ s \in \Z^N; e_N \le s \le n^* \} \supseteq I_{n}$ into $ R  = R_1 \cdot \ldots \cdot R_N$ blocks containing $(2q)^N$ points of the $N$-dimensional integer lattice each. Within each block, there are $2^N$ smaller subsets, which are $N$-dimensional rectangles with all edges of length $q$. Write $I(l,u)$ for the $l$-th subset in the $u$-th block, $l=1,\ldots,2^N$ and $u=1,\ldots, R $. Its cardinality is $q^N$. Moreover, using the requirement on $n$, we have that $R_i \ge 2$ for each $k=1,\ldots,2^N$. Thus, $q^N R \le |I_n|$ and $|I_n|/2^N \le q^N R$. The subcubes $I(l,u)$ have the property that for fixed $l$ the distance between $I(l,u)$ and $I(l,u')$ is at least $q$.

Secondly, we apply recursively (as in \cite{carbon1997kernel}) the lemma of \cite{berbee1979random} to the collection of random variables $Z(l,1),\ldots,Z(l,R)$ for each $l$. We obtain random variables $Z^*(l,1),\ldots,Z^*(l,R) \in \R^{(d+1)q^N}$ with the desired properties. The claim follows now when defining the $X^*(s)$ and the $Y^*(s)$ such that $Z^*(l,u) = ((X^*(s),Y^*(s)):s\in I(l,u))$ $a.s.$ for each $u=1,\ldots,R$ and for each $l=1,\ldots,2^N$.
\end{proof}

The next proposition is a well-known result of \cite{gyorfi} which states sufficient conditions for a consistent estimator. It holds as well for dependent data because in the proof of the proposition those terms which are related to the dependence structure of the data converge to zero by assumption. So it is our task to verify these assumptions later. More precisely, it is assumed that the function classes can approximate the regression function $m$ arbitrarily exactly both in part (a) and in part (b) of this proposition. In our case this assumption does not depend on the data. However, the second requirement in both parts of the proposition is affected by the dependence structure of the data: here it is assumed that a certain empirical mean uniformly converges to the corresponding true mean for each possible function in the sieve. This requirement crucially depends on the data and we can verify this assumption later.

\begin{proposition}[Modified version of \cite{gyorfi} Theorem 10.2]\label{propConsNReg}
Let $\pspace$ be a probability space endowed with the random field $(X,Y)$ which satisfies the model \eqref{lsqI} and Condition \ref{regCond0} such that each $X(s)$ is $\R^d$-valued and each $Y(s)$ is $\R$-valued. Let $\cF_{n(k)} \subseteq L^2\left( \mu_X \right)$ be a class of functions $f: \R^d\rightarrow\R$ for each $k \in\N_+$. Let $( \rho_{n(k)}: k\in \N) \subseteq\R_+$ be an sequence which increases to infinity. Denote the truncated least-squares estimate of $m$ from \eqref{lsqIII} by $\hat{m}_n$. In addition, let the map from \eqref{supMeas2} be $\cA$-$\cB(\R)$-measurable.

(a) If for all $L>0$ both
\begin{align*}
			& \lim_{n \rightarrow \infty} \E{ \inf_{ f\in\cF_n, ||f||_{\infty} \le \rho_n}  \norm{ f- m}_{L^2(\mu_X)} } = 0 \text{ and } \\
		& \lim_{n \rightarrow \infty} \mathbb{E} \Biggl[	\sup_{f \in T_{\rho_n}\cF_n }\Bigg| |I_{n} |^{-1} \sum_{s\in I_{n} } \Big( T_L Y(s) - f(X(s)) \Big)^2 \\
		&\qquad\qquad\qquad\qquad\qquad\qquad - \E{ \Big( T_L Y (e_N) - f(X(e_N ) \Big)^2 } \Bigg| \Biggl] = 0,
\end{align*}
then
$
	\lim_{n \rightarrow\infty} \E{ \int_{\R^d} \left(	\hat{m}_n -m \right)^2 \intd{\mu}_X } = 0
	$.
	
(b) If $\lim_{n\rightarrow\infty} |I_{n} |^{-1} \sum_{s\in I_{n} } |Y(s) - T_L Y(s)|^2 = \E{ |Y(e_N)  - T_L Y(e_N)  |^2 }$ a.s. and if
\begin{align*}
		&\lim_{n \rightarrow \infty} \inf_{ f\in\cF_n, ||f||_{\infty} \le \rho_n}  \norm{ f- m}_{L^2(\mu_X)} = 0 \quad a.s. \text{ and } \\
&\lim_{n \rightarrow \infty} 	\sup_{f \in T_{\rho_n}\cF_n }\Bigg| |I_{n}|^{-1} \sum_{n\in I_{n}} \Big( T_L Y(s) - f(X(s)) \Big)^2 \\
&\qquad\qquad\qquad\qquad - \E{ \Big( T_L Y(e_N)  - f(X(e_N) ) \Big)^2 } \Bigg|  = 0\quad a.s.
\end{align*}
for all $L>0$, then
$
	\lim_{n\rightarrow\infty} \int_{\R^d} \left(	\hat{m}_n -m \right)^2 \intd{\mu}_X = 0$ $a.s.$
\end{proposition}

We proceed with the proof of the first main theorem of Section \ref{Section_NonParaRegGeneral}.
\begin{proof}[Proof of Theorem \ref{ConsistencyTruncatedLeastSquaresGeneral}]
We verify that in both dependence scenarios the sufficient criteria of Proposition~\ref{propConsNReg} are satisfied for the given choices of the parameters. The structure of the proof is quite similar to the one of Theorem 10.3 in \cite{gyorfi}. Therefore we sketch those parts which differ because of the dependence in the data and the assumed covering condition (Condition~\ref{coveringCondition}). The approximation property of the function classes is satisfied by assumption. Moreover, we can assume w.l.o.g.\ that $L < \rho_n$ in both cases because $\rho_n$ tends to infinity. We have to consider the function classes
\begin{align*}
		\cH_n &\coloneqq \big\{ h: \R^d\times\R\rightarrow \R, h(x,y) = |f(x)-T_L(y)|^2 \\
		&\qquad\qquad\qquad\qquad \text{ for all } (x,y) \in \R^d\times\R, \text{ for some } f \in T_{\rho_n}\cF_n \big\}.
\end{align*}
We begin with the case of $\alpha$-mixing data. From Condition~\ref{coveringCondition} we obtain a uniform bound on the $\epsilon$-covering number $\mathsf{N}\big(\epsilon,\cH_n,\norm{\,\cdot\,}_{L^1(\nu)} \big)$ which we denote by $H_{\cH_n}(\epsilon)$, here $\nu$ is an arbitrary probability measure with equal mass concentrated at certain points $z_1,\ldots,z_u \in\R$, $u\in\N_+$. Provided that $L \le \rho_n$, we have for the covering number of this class $\cH_n$ 
\[
		H_{\cH_n} \left(\frac{\epsilon}{32} \right) \le H_{T_{\rho_n} \cF_n} \left(\frac{\epsilon}{32(4\rho_n) } \right) = H_{T_{\rho_n} \cF_n}\left(\frac{\epsilon}{128 \rho_n } \right) = \exp \kappa_{k}(\epsilon/32,\rho_n).
\]
For details on the first inequality, see the proof of Theorem 10.3 in \cite{gyorfi}. Note that the functions in $\cH_n$ are bounded by $4\rho_n^2$ if $L\le \rho_n$. By assumption,
$
	\rho_n^4 \kappa_n(\epsilon/32,\rho_n) \, \left( \log |I_n|\right)^2  \,\big/\, |I_n|^{1/N} \rightarrow 0 \text{ as } n \rightarrow \infty.
$
We use Theorem \ref{USLLNM} to give an upper bound on the following probability (note that we only need to consider the exponential term in \eqref{USLLNMEq0} which decays at a slower rate)
\begin{align}
\begin{split}\label{cRM1}
		&\p\Biggl(	\sup_{f\in T_{\rho_n} \cF_n} \Biggl| 	\frac{1}{|I_{n} | } \sum_{s \in I_{n} } \left|f(X(s))- T_L Y(s)	\right|^2 \\
		&\qquad\qquad\qquad\qquad\qquad\qquad - \E{ \left|f(X(e_N))-T_L Y(e_N)	\right|^2}	\Biggl|	> \epsilon \Biggl) 
		\end{split} \\
		\begin{split}\label{cRM2}
		&\le A_1 \,\exp\left\{ \kappa_n( \epsilon/32, \rho_n) \right\}\,  \exp\left\{	- \frac{A_2 \, |I_n|^{1/N } \, \epsilon^2}{16\rho_n^4 + 4\rho_n^2 \epsilon \left( \log |I_n| \right)^2  } \right\} \\
		& = A_1 \exp \Biggl\{	- \frac{ \epsilon^2\, |I_n|^{1/N }}{16\rho_n^4 + 4\rho_n^2 \epsilon \left( \log |I_n| \right)^2   } \\
		&\qquad\qquad\qquad\qquad \left( A_2 - \frac{ \kappa_n(\epsilon/32,\rho_n) [16\rho_n^4 + 4\rho_n^2 \epsilon \left( \log |I_n| \right)^2] }{ \epsilon^2\, |I_n|^{1/N }   }		\right) \Biggl\},
\end{split}\end{align}
for suitable constants $A_1$ and $A_2$. The weak consistency follows from \eqref{cRM2}: let $\epsilon>0$ be arbitrary but fixed, then
\begin{align*}
		&\E{\sup_{f\in T_{\rho_n} \cF_n} \left| 	\frac{1}{|I_n|} \sum_{s\in I_n} \left|f(X(s))- T_L Y(s)	\right|^2 - \E{ \left|f(X(e_N) )- T_L Y(e_N )	\right|^2}	\right| } \\
		&\le \epsilon + A_1 \exp\left\{ \kappa_n(\epsilon, \rho_n) \right\} \, \int_{\epsilon}^{\infty}  \exp\left\{	- \frac{A_2 \, |I_n|^{1/N } \, t^2}{16\rho_n^4 + 4\rho_n^2 t \left( \log |I_n| \right)^2  } \right\}  \intd{t} \rightarrow \epsilon,
\end{align*}
as $n\rightarrow \infty$. Concerning the $a.s.$ convergence of the estimate, we find that under the condition of $\alpha$-mixing and stationarity the random variables $\{  \left| Y(s) - T_L Y(s)	\right|^2: s\in\Z^N \}$ are ergodic, see Theorem B.4 in \cite{krebs2017orthogonal}. This implies that
$$
 \lim_{n\rightarrow\infty} |I_{n} |^{-1} \sum_{s\in I_{n} } |Y(s) - T_L Y(s)|^2 = \E{ |Y(e_N)  - T_L Y(e_N)  |^2 }\quad a.s.
$$
for all $L>0$. Furthermore, if additionally 
\[
	\rho_n^4  \left( \log |I_n| \right)^4	/ |I_n|^{1/N }  \rightarrow 0 \text{ as } n \rightarrow \infty,
\]
\eqref{cRM2} remains summable for a sequence of index sets $I_{n(k)}$ which satisfies the condition in \eqref{Cond_IndexSet1}. Thus, an application of the Borel-Cantelli Lemma to the same equation yields that the estimator is strongly universally consistent. This finishes the case for $\alpha$-mixing data.

Now consider the case of $\beta$-mixing data. Again, we assume that $\rho_n > L$. Therefore we use the partition of $I_n$ which is provided by Lemma~\ref{CouplingBetaMixing} for the choice $q= \ceil{2 \log |I_n| / c_1 }$. As in \cite{tran1990kernel} we assume that $R_i q = n_i$ for each $i=1,\ldots,N$. We use the coupled random field $(X^*,Y^*)$ to obtain the estimator $\hat{m}^*_n$ of the regression function $m$. We split the integrated error as follows
\begin{align}\label{cRM3}
		\int_{\R^d} |\hat{m}_n - m|^2 \intd{\mu_X} \le 2 \int_{\R^d} |\hat{m}^*_n - m|^2 \intd{\mu_X} + 2 \int_{\R^d} |\hat{m}^*_n - \hat{m}_n|^2 \intd{\mu_X}
\end{align}
Exploiting the properties of $(X^*,Y^*)$, we find that the second term is at most
$$
		\int_{\R^d} |\hat{m}^*_n - \hat{m}_n|^2 \intd{\mu_X} \le 4 \rho_n^2 \1{ (X^*(s),Y^*(s)) \neq (X(s),Y(s)) \text{ for one } s\in I_n }.
$$
Using that $\beta(q) \in \cO( |I_n|^{-2} )$, we have the following bound for the expectation
\begin{align*}
		\E{\int_{\R^d} |\hat{m}^*_n - \hat{m}_n|^2 \intd{\mu_X} } &\le 4 \rho_n^2  \beta(q) \le C \rho_n^2 / |I_n|^2  \rightarrow 0, \quad n\rightarrow \infty,
\end{align*}
where we use that by assumption $\rho_n^4 / |I_n|$ vanishes. Moreover, we have
\begin{align*}
		&\sum_{k=1}^\infty \p\left( \int_{\R^d} |\hat{m}^*_{n(k)} - \hat{m}_{n(k)}|^2 \intd{\mu_X} > \epsilon \right) \\
		&\le 4 \epsilon^{-1} \sum_{k=1}^\infty \rho_{n(k)}^2 \beta(q(n(k)) ) \le C \sum_{k=1}^\infty |I_{n(k)}|^{-1.5} < \infty.
\end{align*}
Hence, $\int_{\R^d} |\hat{m}^*_n - \hat{m}_n|^2 \intd{\mu_X}$ both vanishes in the mean and $a.s.$ Consequently, the first integral in \eqref{cRM3} remains and we need to study the probability in \eqref{cRM1} in this scenario, it equals
\begin{align}
		&\p\Biggl(	\sup_{f\in T_{\rho_n} \cF_n} \Biggl| 	\frac{1}{|I_{n} | } \sum_{l=1}^{2^N} \sum_{u=1 }^R \sum_{s \in I(l,u) } \left|f(X^*(s))- T_L Y^*(s)	\right|^2 \nonumber \\
		&\qquad\qquad\qquad\qquad\qquad\qquad - \E{ \left|f(X^*(e_N))-T_L Y^*(e_N)	\right|^2}	\Biggl|	> \epsilon \Biggl) \nonumber \\
		\begin{split}\label{cRM4}
		&\le \sum_{l=1}^{2^N} \p\Biggl(	\sup_{f\in T_{\rho_n} \cF_n} \Biggl| 	\frac{1}{R}  \sum_{u=1 }^R \Biggl( \sum_{s \in I(l,u) } \left|f(X^*(s))- T_L Y^*(s)	\right|^2 \\
		&\qquad\qquad\qquad\qquad\qquad\qquad - \E{ \left|f(X^*(e_N))-T_L Y^*(e_N)	 \right|^2}	\Biggl)\Biggl|	> \frac{|I_{n} | \epsilon}{2^N R} \Biggl). 
		\end{split}
\end{align}
Using the properties of the coupled process and the stationarity of $(X,Y)$, we see that the summands over the index sets $I(l,1),\ldots,I(l,R)$ are i.i.d.\ for each $l=1,\ldots,2^N$. Consequently, we can apply Theorem 9.1 in \cite{gyorfi}. Note that in the proof of this theorem it is only necessary that the data is independent but not that it is identically distributed. Hence, we obtain for \eqref{cRM4} the bound
\begin{align}
		&2^{N+3} H_{\cH_n} \left(\frac{\epsilon q^N}{2^{N+3}} \right) \exp\left(- \frac{\epsilon^2 |I_n|}{2^{2N+13} q^N \rho_n^4 }		\right) \nonumber \\
		&\le 2^{N+3} H_{T_{\rho_n} \cF_n}  \left(\frac{\epsilon q^N}{2^{N+3} (4\rho_n^2) } \right) \exp\left(- \frac{\epsilon^2 |I_n|}{2^{2N+13} q^N \rho_n^4 }		\right) \nonumber \\
		&= 2^{N+3} \exp\left( \kappa_n\left( \frac{\epsilon q^N }{ 2^{N+3}} , \rho_n \right) - \frac{\epsilon^2 |I_n|}{2^{4N+13}(c_1)^{-N} (\log |I_n|)^N \rho_n^4 }			\right) \nonumber 
\end{align}
using the definition of $\kappa_n$ and of the block size $q$. Note that the factor $q^N$ inside $\kappa_n$ can be neglected as it only decreases $\kappa_n$ marginally if $\epsilon>0$ is fixed. Now, the same computations as in the case of $\alpha$-mixing data yield the result. We do not go into the details.
\end{proof}

The proof of Corollary~\ref{ConsistencyTruncatedLeastSquaresGeneral_Corollary} requires the concept of the Vapnik-Chervonenkis-dimension (VC-dimension). The definition of the VC-dimension is rather technical and can be found in the book of \cite{gyorfi}, Definition 9.6.

\begin{proof}[Proof of Corollary~\ref{ConsistencyTruncatedLeastSquaresGeneral_Corollary}]
It remains to consider some technical issues. Clearly, the map
\[
		\R^{K_n} \times \Omega \ni (a,\omega)\mapsto \sum_{i=1}^{K_n} a_i f_i( X(s,\omega) ) \text{ is $\cB( \R^{K_n} ) \otimes \cA$-measurable}.
\]
The desired measurability of the map in \eqref{supMeas2} follows from the fact that for any measurable function $g$ on a product space $(S \times T, \cS\otimes\cT)$ the set
\begin{align*}
		\left\{ t\in T: \sup_{s\in S} g(s,t) > c \right\} &= \left\{ t\in T\,\big|\, \exists s\in S: g(s,t) > c \right\} \\
		&= \pi^{S\times T}_{T} \{ (s,t)\in S\times T: g(s,t) > c \} \in \cT,
\end{align*}
where $\pi^{S\times T}_{T}$ is the projection from $S \times T$ onto $T$.

Furthermore, the Vapnik-Chervonenkis-dimension is at least 2 if $K_n \ge 2$. Indeed, choose functions $f_1$ and $f_2$. Without loss of generality, there is an $\bar{x}$ in $\R^d$ and an $a$ in $\R$ such that $a f_1 (\bar{x}) = f_2( \bar{x}) > 0$. Since $f_1$ and $f_2$ are linearly independent, exactly one of the following three cases occurs: (1) either there are $x_1$ and $x_2$ in a neighbourhood of $\bar{x}$ such that $a f_1(x_1) > f_2(x_1)$ and $f_2(x_2) > a f_1(x_2)$, (2) or $a f_1 = f_2$ on $U$ and $a f_1 > f_2$ on $\R^d \setminus U$, where $U \subset \R^d$ contains $\bar{x}$, (3) or $f_2 = a f_1$ on $U$ and $f_2 > a f_1$ on $\R^d \setminus U$. In the last two cases we can modify $a$ such that we achieve the first case, by linear independence. Thus, the two points $p_i \coloneqq (x_i, t_i)$ (for $i=1,2$) with the property that $a f_1(x_1) > t_1 > f_2(x_1)$ and $f_2(x_2) > t_2 > a f_1(x_2)$ are shattered by the set of all subgraphs of the linear space $\langle f_1, f_2 \rangle$, hence,	$\cV_{ \langle f_1, \ldots, f_n \rangle^+ } \ge \cV_{ \langle f_1, f_2 \rangle^+ } \ge 2$. Consequently, the conditions of Theorem \ref{boundCoveringNumber} are satisfied. We have
\begin{align*}
		\kappa_n (\epsilon, \rho_n) = \log H_{T_{\rho_n}\cF_n} \left( \frac{\epsilon}{4\rho_n} \right) &\le \log \left(	3 \left( \frac{ 16 e\rho_n^2}{\epsilon} \,\log \frac{ 24 e  \rho_n^2}{\epsilon} \right)^{\cV_{ \left(T_{\rho_n} \cF_n \right)^+}} \right) \\
		&\le \log 3 + (K_n+1) \log \left( (24)^2 \left(\frac{e}{\epsilon} \right)^2 \rho_n^4 \right) \\
		&= \cO(K_n \log \rho_n).
\end{align*}
The statement follows now from Theorem~\ref{ConsistencyTruncatedLeastSquaresGeneral}. 
\end{proof}

We need another proposition and a piece of notation to prove the rate of convergence of the regression estimator.
\begin{notation}
Let $f$ be a real-valued function on $\R^d$ and let the distribution of the $X(s)$ be given by $\mu_X$. Let $X'=\{ X'(s): s\in \Z^N \}$ be an i.i.d.\ ghost sample with the same marginals as $X$. Moreover, $X^*$ is constructed for each $n\in\N_+^N$ as in Lemma~\ref{CouplingBetaMixing} and the random field $X^\dagger$ is an independent copy of $X^*$. Define the following empirical $L^2$-norms
\begin{align*}
		&\norm{f}_{|I_n|} \coloneqq \left( |I_n|^{-1} \sum_{s\in I_n} f(X(s))^2 \right)^{1/2}, \quad \norm{f}^{'}_{|I_n|} \coloneqq \left( |I_n|^{-1}\sum_{s\in I_n} f(X^{'}(s) )^2 \right)^{1/2} \\
		&\text{ and } \norm{f}^{\sim}_{|I_n|} \coloneqq \left( (2|I_n|)^{-1}\sum_{s\in I_n} f(X(s))^2+f(X^{'}(s) )^2 \right)^{1/2} \\
		&\text{ as well as } \norm{f}^*_{|I_n|} \coloneqq \left( |I_n|^{-1} \sum_{s\in I_n} f(X^*(s))^2 \right)^{1/2} \\
		&\text{ and } \norm{f}^\dagger_{|I_n|} \coloneqq \left( |I_n|^{-1} \sum_{s\in I_n} f(X^\dagger(s))^2 \right)^{1/2}.
\end{align*}

Consider the random point measure $\nu$ with equal masses which is induced by the sample of the random field and of the ghost sample $(X(I_n),X'(I_n)$, i.e., $\nu = (2|I_n|)^{-1} \sum_{s\in I_n} \big(\delta_{ X(s) } + \delta_{X'(s)} \big)$. We abbreviate the $\epsilon$-covering number of a function class $\cG$ w.r.t.\ 2-norm of $\nu$ by
$$ \mathsf{N}_2\left(\epsilon, \cG, (X(I_n), X^{'} (I_n)) \right) \coloneqq \mathsf{N} \left( \epsilon,\cG, \norm{\,\cdot\,}_{L^2(\nu)}		\right). $$
\end{notation}

The next two statements prepare the second main theorem of Section \ref{Section_NonParaRegGeneral} which is Theorem~\ref{nonParaRegRateOfConvergence}. The first is intended for $\alpha$-mixing data, the second for $\beta$-mixing data.

\begin{proposition}\label{empNormAlphaMixing}
Assume that the random field $X$ satisfies Condition \ref{regCond0} \ref{Cond_AlphaMixing}. Let $\cG$ be a class of $\R$-valued functions on $\R^d$ which are all bounded by a universal constant $B$. Then for all $\epsilon>0$
\begin{align}
\begin{split}\label{empNormAlphaMixingEq0}
	&\p\left( \sup_{ f \in \cG}\; \norm{f} - 2 \norm{f}_{I_n} > \epsilon \right) \\
	&\le A_1 \norm{ \mathsf{N}_2\left( \frac{\sqrt{2} \epsilon}{32 }, \cG, (X(I_n), X^{'} (I_n)) \right) }_{\p,\infty} \\
	&\quad  \cdot \left( \exp\left( - \frac{A_2 \epsilon^4 |I_n|^{1/N }} {B^4 + B^2\,\epsilon^2 \left(\log |I_n| \right)^2 } \right) + \exp\left( - \frac{ A_3 \epsilon^2 |I_n| }{ B^2} \right) \right) 
\end{split}
\end{align}
for constants $0<A_1,A_2,A_3 < \infty$ which neither depend on the bound $B$, nor on $\epsilon$, nor on the index set $I_n$.
\end{proposition}

Provided that the Vapnik-Chervonenkis dimension $\cV_{\cG^+}$ is at least 2 and that $\epsilon$ is sufficiently small, the bound from Proposition \ref{empNormAlphaMixing} is non-trivial: we have with Proposition~\ref{boundCoveringNumber}
\begin{align*}
		&\log \norm{ \mathsf{N}_2\left( \frac{\epsilon}{16 \sqrt{2}}, \cG, (X(I_n), X^{'} (I_n)) \right) }_{\p,\infty} \\
		&\le \log 3 + \cV_{\cG^+} \log \left( \frac{ 16^3 e B^2}{\epsilon^2} \cdot \log \frac{ 24 \cdot 16^2 e B^2}{\epsilon^2} \right).
\end{align*}

\begin{proof}[Proof of Proposition~\ref{empNormAlphaMixing}]
Let $\{X(s): s\in I_n \}$ be a subset of the strongly mixing and stationary random field $X$ and let $\{X^{'}(s): s\in I_n\}$ be the corresponding ghost sample. We use the relation
\begin{align*}
		&\p\left( \sup_{ f \in \cG}\; \norm{f} - 2 \norm{f}_{I_n} > \epsilon \right) \\
		&\le \p\left( \sup_{ f \in \cG}\; \norm{f} - 2 \norm{f}'_{I_n} > \frac{\epsilon}{2} \right) + \p\left(\sup_{ f \in \cG}\; \norm{f}'_{I_n} -  \norm{f}_{I_n} > \frac{\epsilon}{4} \right).
\end{align*}

We only consider the second probability on the right-hand-side of the last inequality, bounds on the first probability are given by the second term in the second line of \eqref{empNormAlphaMixingEq0} and are derived in Theorem 11.2 of \cite{gyorfi}. Let $U_1,\ldots,U_{H^*}$ be an $\epsilon/(16 \sqrt{2})$-covering of $\cG$ with respect to the empirical $L^2$-norm of the sample $ \big (X(I_n), X^{'}(I_n) \big)$ with the definition $H^* \coloneqq  \mathsf{N}_2 \big( \epsilon /(16 \sqrt{2}), \cG, \big(X(I_n), X^{'}(I_n) \big) \big)$ and $U_k \coloneqq \{f\in \cG: \norm{f - g_k}^{\sim}_{I_n} < \epsilon/(16 \sqrt{2} ) \}$, where the covering functions are $g_1,\ldots g_{H^*}$. Note that $H^*$ and the $U_k$ are random and that both $\norm{\,\cdot\,}_{I_n}$ and $\norm{\,\cdot\,}'_{I_n}$ are bounded by $\sqrt{2} \norm{\,\cdot\,}^{\sim}_{I_n}$. Then,
\begin{align}
	\p\left(\exists f\in \cG: \norm{f}_{I_n}^{'} - \norm{f}_{I_n} > \frac{\epsilon}{4}		\right) &\le \sum_{k=1}^{||H^*||_{\p,\infty}} \p\left( \exists f\in U_k: \norm{f}_{I_n}^{'} - \norm{f}_{I_n} > \frac{\epsilon}{4}	 \right). \label{empNorm2}
\end{align}
Now, we use that $\norm{f}_{I_n} \le \sqrt{2} \norm{f}_{I_n}^{\sim}$ to obtain for $f \in U_k$ the inequality
\begin{align*}
		\norm{f}_{I_n}^{'} - \norm{f}_{I_n} &= \norm{f}_{I_n}^{'} - \norm{g_k}_{I_n}^{'} + \norm{g_k}_{I_n}^{'} - \norm{g_k}_{I_n} + \norm{g_k}_{I_n} - \norm{f}_{I_n}  \\
		&\le \norm{f - g_k}_{I_n}^{'} + \left( \norm{g_k}_{I_n}^{'} - \norm{g_k}_{I_n}\right) + \norm{f - g_k}_{I_n} \\
		&\le 2 \sqrt{2}\, \frac{\epsilon}{16 \sqrt{2}} + \left( \norm{g_k}_{I_n}^{'} - \norm{g_k}_{I_n}\right). 
\end{align*}
Hence, $\big\{ \exists f \in U_k:  \norm{f}_{I_n}^{'} - \norm{f}_{I_n} > \frac{\epsilon}{4} \big\}$ is a subset of $\big\{\norm{g_k}_{I_n}^{'} - \norm{g}_{I_n} > \frac{\epsilon}{8} \big\}$. Since the inequality $a - b > c$ implies $a^2 - b^2 > c^2$ for $a,b,c \ge 0$, we get for the probabilities on the right-hand-side of \eqref{empNorm2} the following bounds
\begin{align} 
  &\p\left(	\norm{g_k}_{I_n}^{'} - \norm{g_k}_{I_n} > \frac{\epsilon}{8} \right) \le \p\left(	\left(\norm{g_k}_{I_n}^{'}\right)	^2 - \left(\norm{g_k}_{I_n} \right)^2	> \frac{\epsilon^2}{64} \right) \nonumber \\
	&\le \p\Biggl(	\frac{1}{|I_n|} \sum_{s\in I_n} \left\{ g_k( X^{'}(s) )^2 - \E{ g_k( X^{'}(e_N))^2} \right\} \nonumber \\
	&\quad- \frac{1}{|I_n|} \sum_{s\in I_n} \left\{ g_k( X(s) )^2 - \E{ g_k( X(e_N))^2} \right\} > \frac{\epsilon^2}{64} \Biggl) \nonumber \\
	\begin{split} 
	&\le \p\left( \left|	\frac{1}{|I_n|} \sum_{s\in I_n} g_k( X^{'}(s) )^2 - \E{ g_k( X^{'}(e_N) )^2} \right| > \frac{\epsilon^2}{128} \right)  \\
	&\quad + \p\left( \left|	\frac{1}{|I_n|} \sum_{s\in I_n} g_k( X(s) )^2 - \E{ g_k( X(e_N) )^2} \right| > \frac{\epsilon^2}{128} \right). \label{empNorm3}
\end{split} \end{align}
The first term from \eqref{empNorm3} can be bounded by Hoeffding's inequality, we have
\begin{align} \label{empNorm4}
		\p\left( \left|	\frac{1}{|I_n|} \sum_{s\in I_n} g_k( X^{'} (s) )^2 - \E{ g_k( X^{'}(e_N) )^2} \right| > \frac{\epsilon^2}{128} \right) \le 2 \exp\left( - C \epsilon^4 \frac{ |I_n| }{B^4}		\right).
\end{align}
We apply Proposition \ref{applBernstein} to the second term and obtain that
\begin{align} 
		&\p\left( \left|	\frac{1}{|I_n|} \sum_{s\in I_n} g_k( X(s) )^2 - \E{ g_k( X(s))^2} \right| > \frac{\epsilon^2}{128} \right)  \nonumber\\
		&\le \exp\left( - \frac{C \epsilon^4 |I_n|^{1/N }} {B^4 + B^2\,\epsilon^2 \left(\log |I_n| \right)^2 } \right). \label{empNorm5}
\end{align}
Obviously, the bound in \eqref{empNorm5} dominates the bound in \eqref{empNorm4}. This finishes the proof.
\end{proof}

The next proposition is a generalization of Theorem 11.2 of \cite{gyorfi} for $\beta$-mixing data.

\begin{proposition}\label{empNormBetaMixing}
Assume that the random field $X$ satisfies Condition \ref{regCond0} \ref{Cond_BetaMixing}. Let $\cG$ be a class of $\R$-valued functions on $\R^d$ which are all bounded by $B\in\R_+$. Let $n$ be sufficiently large such that both $8/c_1 \log |I_n| < \min\{n_i:i=1,\ldots,N\}$ and $C^* \le 2^{3N-1} (c_1)^{-N} ( \log |I_n|)^N$ where the constant $C^*$ is defined in \eqref{empNormBetaMixingEq2aa}. Then for all $\epsilon>0$
\begin{align}
\begin{split}\label{empNormBetaMixingEq0}
	\p\left( \sup_{ f \in \cG}\; \norm{f} - 2 \norm{f}^*_{I_n} \ge \epsilon \right) &\le   3 \cdot 2^{N}  \norm{ \mathsf{N}_2\left( \frac{\sqrt{2} \epsilon}{32}, \cG, (X(I_n), X^{'} (I_n)) \right) }_{\p,\infty}	\\
	&\quad \cdot\exp \left(- \frac{\epsilon^2 |I_n| }{2^{5N+5} B^2 c_1^{-N} (\log |I_n| )^N } \right)
\end{split}
\end{align}
\end{proposition}

\begin{proof}[Proof of Proposition~\ref{empNormBetaMixing}]
Set $q=\ceil{2/c_1 \log |I_n| }$ and apply Lemma~\ref{CouplingBetaMixing}. We obtain a partition of $I_n$ given by $\{ I(l,u): l=1,\ldots,2^N,u=1,\ldots,R\}$ such that we can write
\begin{align*}
&|I_n|^{-1}\sum_{s\in I_n } f(X^*(s))^2 = R^{-1} \sum_{l=1}^{2^N} \sum_{u=1}^R  Z^*(l,u)^2\quad  \\
&\text{ and } \quad |I_n|^{-1}\sum_{s\in I_n } f(X^\dagger(s))^2 = R^{-1} \sum_{l=1}^{2^N} \sum_{u=1}^R Z^\dagger(l,u)^2,
\end{align*}
where 
\begin{align*}
		&Z^*(l,u)= \left( R |I_n|^{-1} \sum_{s\in I(l,u) } f(X^*(s))^2 \right)^{1/2} \\
		&\text{ and } Z^\dagger(l,u)= \left( R |I_n|^{-1} \sum_{s\in I(l,u) } f(X^\dagger(s))^2 \right)^{1/2}.
\end{align*}
Note that $0\le Z^*(l,u), Z^\dagger(l,u) \le B$.

In the following, let $\tilde{f}$ be a function in $\cG$ such that $\norm{\tilde{f}} - 2\norm{\tilde{f}}_{I_n} \ge \epsilon$ if there is such a function. Otherwise, $\tilde{f}$ is any other function. We write $\p^*$ for the conditional probability measure and $\mathbb{E}^*$ for the conditional expectation given the data $X^*(I_n)$. 

The remaining proof is a modification of Theorem 11.2 in \cite{gyorfi} and is split in three steps. In the first step, we show that 
\begin{align}\label{empNormBetaMixingEq1}
		\p\left( \sup_{ f \in \cG}\; \norm{f} - 2 \norm{f}^*_{I_n} \ge \epsilon \right) \le \frac{3}{2} \p\left( \sup_{ f \in \cG}\; \norm{f}^\dagger_{I_n} - \norm{f}^*_{I_n} \ge \frac{\epsilon}{4} \right), 
\end{align}
if $B^2/\epsilon^2 \le  |I_n| / (2^{2N+6} C^* )$ where
\begin{align}\label{empNormBetaMixingEq2aa}
		C^*\coloneqq \sqrt{2} \sqrt{1+C_\p} (1+C_N \bar\beta_\infty) .
\end{align}
Here $C_\p$ is a uniform bound of the essential suprema of the Radon-Nikod{\'y}m derivatives in \eqref{BoundRadonNikodymDerivative} and the factor $\bar\beta_\infty$ equals $\sum_{k=0}^\infty k^{N-1} \sqrt{\beta(k)} < \infty$; additionally, the constant $C_N$ depends on the lattice dimension $N$ and is given below.

For this result, we need that 
\begin{align}\label{empNormBetaMixingEq2a}
		\p^*\left(  2\norm{\tilde{f}}^\dagger_{I_n} + \frac{\epsilon}{2} \ge \norm{\tilde{f}} \right) \ge 1 - \p^*\left( 3\norm{\tilde{f}}^2 + \frac{\epsilon^2}{4} \le 4\left( \norm{\tilde{f}}^2-\left(\norm{\tilde{f}}^\dagger_{I_n}\right)^2 \right) \right)
\end{align}
Indeed, this follows with some calculations (see the proof of Theorem 11.2 \cite{gyorfi}). Furthermore, we need a result, which follows using the $\beta$-mixing property and a lemma in \cite{krebs2018large}, 
\begin{align}
	& \sum_{s,t\in I(l,u) } \Ec{ \tilde{f}(X^\dagger(s))^2 \tilde{f}(X^\dagger(t))^2 } \nonumber \\
		&\le \sqrt{2} \sqrt{1+C_\p} \Ec{ \tilde{f}(X^\dagger(e_N))^4 } \sum_{s,t\in I(l,u) } \beta(\norm{s-t}_{\infty})^{1/2} \nonumber \\
		&\le \sqrt{2} \sqrt{1+C_\p} (1+C_N \bar\beta_\infty) B^2 \norm{\tilde{f}}^2  q^N, \nonumber
\end{align}
for a certain constant $C_N$ which depends on the lattice dimension $N$.

Moreover, using that for a fixed $l$ the blocked random variables $\{X^\dagger( I(l,u)):u=1,\ldots,R\}$ are independent, the probability on the right-hand-side of \eqref{empNormBetaMixingEq2a} is at most
\begin{align}
		&\frac{16}{|I_n|^2} \frac{ \operatorname{Var}^*\left(  \sum_{l,u} \sum_{s\in I(l,u) } \tilde{f}(X^\dagger(s))^2  \right) }{\left(3\norm{\tilde{f}}^2 + \frac{\epsilon^2}{4} \right)^2} \nonumber \\
		&\le \frac{2^{N+4} }{|I_n|^2} \sum_{l=1}^{2^N} \sum_{u=1}^R \frac{ \Ec{ \left(\sum_{s\in I(l,u) } \tilde{f}(X^\dagger(s))^2 \right)^2 } }{\left(3\norm{\tilde{f}}^2 + \frac{\epsilon^2}{4} \right)^2} \nonumber \\
		&\le \frac{2^{N+4} }{|I_n|^2} \frac{2^N R C^* B^2 \norm{\tilde{f}}^2 q^N }{\left(3\norm{\tilde{f}}^2 + \frac{\epsilon^2}{4} \right)^2}\le \frac{2^{2N+6} C^* B^2}{3|I_n| \epsilon^2}. \label{empNormBetaMixingEq2b}
\end{align}
This last term is at most $1/3$ if $|I_n|  \ge 2^{2N+6} C^* B^2 / \epsilon^2$. In particular, the right-hand-side of \eqref{empNormBetaMixingEq2a} is then at least $2/3$.

Using once more a result of \cite{gyorfi}, we have that 
\begin{align*}
		&\p\left( \sup_{ f \in \cG}\; \norm{f}^\dagger_{|I_n|} - \norm{f}^*_{|I_n|} \ge \frac{\epsilon}{4} \right) \\
		&\ge \E{\1{\norm{\tilde{f}} - 2\norm{\tilde{f}}^{*}_{|I_n|} \ge \epsilon } \p^*\left(  2\norm{\tilde{f}}^\dagger_{|I_n|} + \frac{\epsilon}{2} \ge \norm{\tilde{f}} \right)} 
\end{align*}
Consequently, \eqref{empNormBetaMixingEq1} follows from this last inequality if $|I_n| \ge 2^{2N+6} C^* B^2 / \epsilon^2$.

In the second step, consider an $\epsilon/(16 \sqrt{2})$-covering of $\cG$ with respect to the empirical $L^2$-norm of the sample $ \big (X^*(I_n), X^\dagger(I_n) \big)$. It follows as in the proof of Proposition~\ref{empNormAlphaMixing} that
\begin{align}
	\p\left(\sup_{ f \in \cG}\; \norm{f}_{I_n}^{\dagger} - \norm{f}^*_{I_n} \ge \frac{\epsilon}{4}		\right) &\le \sum_{k=1}^{||H^*||_{\p,\infty} } \p\left(	\norm{g_k}_{I_n}^{\dagger} - \norm{g_k}^*_{I_n} \ge \frac{\epsilon}{8} \right), \label{empNormBetaMixingEq2}
	\end{align}
	where $||H^*||_{\p,\infty} \le \norm{ \mathsf{N}_2\left( \frac{\sqrt{2} \epsilon}{32}, \cG, (X(I_n), X^{'} (I_n)) \right) }_{\p,\infty}	$.
	
Consequently, it remains to bound the last probability in \eqref{empNormBetaMixingEq2}. This is done in the third step. Consider a function $f$ such that $|f|\le B$, then
\begin{align}
	&\p\left(\norm{f}^\dagger_{I_n} - \norm{f}^*_{I_n} \ge  \frac{\epsilon}{8} \right) \nonumber \\
	 &= \p\left( \Big( R^{-1} \sum_{l,u}  Z^\dagger(l,u) ^2 \Big)^{1/2} - \Big(R^{-1} \sum_{l,u} Z^*(l,u)^2 \Big)^{1/2} \ge \frac{\epsilon}{8} \right) \nonumber \\
&=\p\left( \frac{ R^{-1} \sum_{l,u} Z^\dagger(l,u)^2 - Z^*(l,u)^2 } {\Big( R^{-1} \sum_{l,u} Z^\dagger(l,u)^2 \Big)^{1/2} + \Big( R^{-1} \sum_{l,u} Z^*(l,u)^2 \Big)^{1/2}} \ge  \frac{\epsilon}{8} \right) \nonumber \\
&\le \sum_{l=1}^{2^N} \p\left( \frac{ R^{-1} \sum_{u=1}^R Z^\dagger(l,u)^2 - Z^*(l,u)^2 } {\Big( R^{-1} \sum_{u=1}^R Z^\dagger(l,u)^2 \Big)^{1/2} + \Big( R^{-1} \sum_{u=1}^R Z^*(l,u)^2 \Big)^{1/2}} \ge \frac{\epsilon}{2^{N+3} } \right) \nonumber \\
&=  \sum_{l=1}^{2^N} \p\left( \Big( R^{-1} \sum_{u=1}^R  Z^\dagger(l,u)^2 \Big)^{1/2} - \Big( R^{-1} \sum_{u=1}^R  Z^*(l,u)^2 \Big)^{1/2} \ge \frac{\epsilon}{2^{N+3} }  \right). \label{empNormBetaMixingEq3}
\end{align}
Next, we use a trick which induces additional randomness and which can be applied to the last probabilities. W.l.o.g.\, we consider the case $l=1$. Then, choose i.i.d.\ random variables $V(1),\ldots,V(R)$ which are uniformly distributed on $\{-1,1\}$ and define
\begin{align*}
U^\dagger(u) \coloneqq 
\begin{cases}
Z^\dagger(1,u) & \text{ if } V(u) = 1 \\
Z^*(1,u) & \text{ if } V(u) = -1
\end{cases}
\quad \text{ and }\quad 
U^*(u) \coloneqq 
\begin{cases}
Z^*(1,u)& \text{ if } V(u) = 1 \\
Z^\dagger(1,u)& \text{ if } V(u) = -1.
\end{cases}
\end{align*}
As the $Z^*(1,u)$ and $Z^\dagger(1,u)$ are independent and have for each $u$ identical distributions, we can replace their distribution with the distribution of the $U^*(u)$ and $U^\dagger(u)$. Now, write $\p^*$ for the probability measure conditioned on $\sigma(Z^\dagger(1,u), Z^*(1,u), u=1,\ldots R)$. Then if $l=1$, the probability in \eqref{empNormBetaMixingEq3} equals
\begin{align*}
&\p\left( \Big( R^{-1} \sum_{u=1}^R  U^\dagger(u)^2 \Big)^{1/2} - \Big( R^{-1} \sum_{u=1}^R  U^*(u)^2 \Big)^{1/2} \ge \frac{\epsilon}{2^{N+3} }  \right) \\
&= \mathbb{E}\Bigg[ \p^*\Bigg( R^{-1} \sum_{u=1}^R V(u) ( Z^\dagger(1,u)^2 - Z^*(1,u)^2 ) \\
&\qquad\qquad \ge \frac{\epsilon}{2^{N+3} }  \left\{ \Big( R^{-1} \sum_{u=1}^R  Z^\dagger(1,u)^2 \Big)^{1/2} + \Big( R^{-1} \sum_{u=1}^R  Z^*(1,u)^2 \Big)^{1/2} \right\} \Bigg]. 
\end{align*}
Due to the independence between the $V(u)$ and the $(Z^\dagger(1,u), Z^*(1,u))$, we can bound the inner conditional probability with Hoeffding's inequality and obtain the bound
\begin{align}
&2 \exp\left(- \frac{R \epsilon^2}{2^{2N+5}} \frac{\sum_{u=1}^R  Z^\dagger(1,u)^2 + Z^*(1,u)^2}{\sum_{u=1}^R | Z^\dagger(1,u)^2  - Z^*(1,u)^2|^2 } \right) \nonumber \\
 &\le 2 \exp\left(- \frac{\epsilon^2 |I_n| }{2^{3N+5} B^2 q^N } \right) \nonumber \\ 
&\le 2 \exp \left(- \frac{\epsilon^2 |I_n| }{2^{5N+5} B^2 c_1^{-N} (\log |I_n| )^N } \right). \label{empNormBetaMixingEq4}
\end{align}
We use for the last inequality the three relations $R\ge |I_n|/(2^N q^N )$ and $q^N\le  2^{2N}/c_1^N (\log |I_n| )^N$ as well as
$$
| Z^\dagger(1,u)^2  - Z^*(1,u)^2|^2 \le Z^\dagger(1,u)^4 + Z^*(1,u)^4 \le B^2  ( Z^\dagger(1,u)^2 + Z^*(1,u)^2).
$$
Combining \eqref{empNormBetaMixingEq1} to \eqref{empNormBetaMixingEq4} yields the result given in \eqref{empNormBetaMixingEq0} if $|I_n| \ge 2^{2N+6} C^* B^2 / \epsilon^2$. Otherwise in the case that $|I_n| < 2^{2N+6} C^* B^2 / \epsilon^2$, the exponential in \eqref{empNormBetaMixingEq4} is at least $e^{-1}$ if $C^* \le 2^{3N-1} (c_1)^{-N} ( \log |I_n|)^N$, hence, the right-hand-side of \eqref{empNormBetaMixingEq0} is greater than one; so the inequality is also true in this case.
\end{proof}

\begin{proof}[Proof of Theorem~\ref{nonParaRegRateOfConvergence}]
We begin with the case of $\alpha$-mixing data and use the decomposition
\begin{align}
		&\int_{\R^d} | \hat{m}_n - m |^2 \intd{\mu_X}		\nonumber \\
		&= \norm{ \hat{m}_n - m }^2  = \left(\norm{ \hat{m}_n - m } - 2\norm{ \hat{m}_n - m}_{I_{n} } + 2\norm{ \hat{m}_n - m}_{I_{n} } \right)^2 \nonumber \\
		&\le  2\,\max \left( \norm{ \hat{m}_n - m } - 2\norm{ \hat{m}_n - m}_{I_{n} } , 0 \right) ^2 + 8 \left( \norm{ \hat{m}_n - m}_{I_{n} } \right)^2 \label{convRate1}
\end{align}
The exponentially decreasing mixing rates ensure that the norm of the conditional covariance matrix remains bounded and that we can use Theorem 11.1 of \cite{gyorfi} even in the case where the error terms $\epsilon(s)$ are not uncorrelated. There is a constant $C_1$ such that $\norm{ \operatorname{Cov}( Y(I_{n} ) \,|\, X(I_{n } ) ) }_2 \le C_1$ for all $k \in\N$. Indeed, consider the operator norms for matrices which are defined for a matrix $A \in \R^{u_1 \times u_2}$ and $p\in[1,\infty]$ by the corresponding $p$-norm on $\R^{u_1}$ (resp.  on $\R^{u_2}$) as $\norm{A}_{p} = \max_{x\in\R^{u_2}: \norm{x}_p =1} \norm{Ax}_p$. We have the norm inequality $\norm{A}_2 \le \sqrt{ \norm{A}_1 \, \norm{A}_{\infty} }$. As the covariance matrix is symmetric, the $\infty$- and the 1-norm are equal. We consider a line (resp. a column) of the covariance matrix that contains the conditional covariances of the $Y(s)$. By assumption, the error terms satisfy $\E{ |\epsilon(s)|^{2+\gamma}}<\infty$ for some $\gamma > 0$. We use Davydov's inequality from Appendix~\ref{davydovsIneq} and the bound on the mixing coefficients, $\alpha(k) \le c_0 \exp(- c_1 k)$ for certain $c_0,c_1 \in \R_+$. We obtain
\begin{align*}
		&\sum_{t\in I_{n} } |\operatorname{Cov}( Y(s), Y(t) \,|\,X(I_{n }) )| \\
		&\le \norm{\varsigma}_{\infty}^2 \sum_{t \in I_{n }}  |\operatorname{Cov}( \epsilon(s), \epsilon(t) ) | \\
		&\le 10 \norm{\varsigma}_{\infty}^2  \E{ |\epsilon(s)|^{2+\gamma}}^{2/(2+\gamma)} \sum_{t \in I_{n } } \alpha( \norm{s-t}_{\infty} )^{\gamma/(2+\gamma)} \\
		&\le 10 \norm{\varsigma}_{\infty}^2 c_0 \E{ |\epsilon(s)|^{2+\gamma}}^{2/(2+\gamma)}  \\
		&\quad \times \sum_{u=0}^{\max_{ 1\le i \le N} n_i } \left( (2u+1)^N - (2u-1)^N \right) \, \exp\left( - c_1\, \frac{\gamma}{2+\gamma} u \right) \\
		&\le C_1,
\end{align*}
for a universal constant $C_1<\infty$ and for all $s\in I_{n }$. Hence, 
$	\norm{ \operatorname{Cov}( Y(I_{n }) \,|\, X(I_{n }) ) }_2 \le C_1 $.	Thus, we find with Theorem 11.1 of \cite{gyorfi}, which is applicable to dependent data too, that
\begin{align}
\E{ \norm{\hat{m}_n - m}_{I_{n } }^2 } \le \ C_1 \frac{ K_n }{ |I_n| } + \inf_{f \in \cF_n} \int_{\R^d} \left(	f- m	\right)^2   \intd{\mu}_X. \label{convRate2}
\end{align}
Next, consider the expectation of the first term in \eqref{convRate1}, it admits the upper bound
\begin{align}
		&\E{ \left\{\max \left( \norm{ \hat{m}_n - m } - 2\norm{ \hat{m}_n - m}_{I_n} , 0 \right) \,\right\}^2 } \nonumber \\
		&\le v + \int_{v}^{\infty} \p\left( \left\{\max \left( \norm{ \hat{m}_n - m } - 2\norm{ \hat{m}_n - m}_{I_n} , 0 \right) \,\right\}^2 > u \right) \intd{u} \nonumber \\
		&\le v + \int_{v}^{\infty} \p\left(	\exists f \in T_L \cF_n: \norm{f-m} - 2\norm{f-m}_{I_{n } } > \sqrt{ u } \right) \intd{u}, \label{convRate3}
\end{align}
for each $v>0$ and if $|I_n|$ is large enough.

We apply Proposition~\ref{empNormAlphaMixing}; note that the second exponential term in \eqref{empNormAlphaMixingEq0} is negligible, so we only consider the first term here. We find with Proposition \ref{boundCoveringNumber} that the covering number is in $\cO\left( \left( L^2/v \right)^{2(K_n+1)  }\right)$ provided that $v < 16^2 L^2$; w.l.o.g.\ this is the case. Hence, \eqref{convRate3} can be bounded by
\begin{align}\label{convRate4}
		&v + A_1 \left( \frac{L^2}{v} \right)^{2 (K_n + 1)} \int_{v}^{\infty} \exp\left( - \frac{A_2 u^2 |I_n|^{ 1/N }} {L^4 + L^2\,u \left(\log |I_n|\right)^2 } \right) \intd{u}.
\end{align}
Define $v \coloneqq K_n \log(|I_n|) / |I_n|^{1/(2N)}$ which converges to zero by assumption. One finds that \eqref{convRate4} is in $\cO(v)$. Combining this result with \eqref{convRate1} and \eqref{convRate2} implies the assertion for $\alpha$-mixing data.

Next, we consider the case of $\beta$-mixing data. For that reason we need the coupled process $(X^*,Y^*)$ obtained from Lemma~\ref{CouplingBetaMixing} for $q=\ceil{2/c_1 |I_n| }$. We also compute the truncated least-squares estimate for the coupled regression problem and denote it by $\hat{m}_n^*$. Then the following upper bound of \eqref{convRate1} is true in terms of the estimate $\hat{m}_n^*$
\begin{align}
\begin{split}
		&4 \left\{\max \left( \norm{ \hat{m}^*_n - m } - 2\norm{ \hat{m}^*_n - m}_{I_n} , 0 \right) \,\right\}^2   \\
		&\qquad\qquad + 16 \norm{\hat{m}_n - m}^2_{I_n} + 16 \norm{\hat{m}_n - \hat{m}^*_n }^2_{I_n} +  2\norm{\hat{m}_n - \hat{m}^*_n }^2 . \label{convRate5}
\end{split}\end{align}
Consider the expectation of the last two terms: we have
$$
	\E{ \norm{\hat{m}_n - \hat{m}^*_n }^2_{I_n} }  + \E{ \norm{\hat{m}_n - \hat{m}^*_n }^2 } \le 2^{N+3} L^2 R \beta(q) = o(|I_n|^{-1}),
	$$
	this follows from the coupling property. Consequently, these two terms are negligible. A bound on the expectation of the second term in \eqref{convRate5} has already been established in \eqref{convRate2}. We can bound the expectation of the first term in \eqref{convRate5} similar as in the case of $\alpha$-mixing data in \eqref{convRate3} but this time using Proposition~\ref{empNormBetaMixing}. Thus, instead of \eqref{convRate4} and if $|I_n|$ is large enough, we obtain for the expectation of this term the bound 
\begin{align}\begin{split}\label{convRate6}
			&v + A_1 \left( \frac{L^2}{v} \right)^{2 (K_n + 1)} \int_{v}^{\infty} \exp \left(- \frac{ u \, |I_n| }{2^{5N+5} L^2 c_1^{-N} (\log |I_n| )^N } \right) \intd{u},
\end{split}\end{align}
where $v$ is positive and again $A_1$ is a positive constant. One finds that for the choice $v = K_n (\log |I_n| )^{N+2} / |I_n|$ both terms in \eqref{convRate6} are in $\cO(v)$. This proofs the result in the case of $\beta$-mixing data.
\end{proof}

We come to the proofs of the theorems in Section~\ref{Section_NonParaRegWavelets}. 

\begin{proof}[Proof of Theorem~\ref{WaveletsDenseLp}]
If $\bigcup_{j\in\Z} U_j$ is not dense in $L^{p}(\mu)$, there is a $0 \neq g \in L^q(\mu)$ which satisfies $\int_{\R^d} fg \, \intd{\mu} = 0$ for all $f\in \overline{\bigcup_{j\in\Z} U_j}$ where $q$ is H{\"o}lder conjugate to $p$. We show that the Fourier transform of $g$ is zero which contradicts the assumption that $g \neq 0$. This proves in particular that $\bigcup_{j\in \Z} U_j$ is dense. Consider the Fourier transform of this element $g$ which we define here for reasons of simplicity as
\begin{align*}
		\cF g \colon \R^d \rightarrow \C,\, \xi \mapsto \int_{\R^d} g(x)\, e^{ i \langle x, \xi \rangle}\, \mu(\mathrm{d}x),
\end{align*}
where $\scalar{\cdot}{\cdot}$ is the Euclidean inner product on $\R^d$.

Since the scaling function $\Phi$ is of the form $\Phi = \otimes_{i=1}^d \phi$ and $\phi$ is a compactly supported one-dimensional scaling function, we can assume that the support of $\Phi$ is contained in the cube $[0,A]^d$ for some $A \in \N_+$. Choose $1 > \epsilon > 0$ arbitrary, there is an $n \in \N$ such that we have for $Q \coloneqq [-A n,\, A n]^d$
\[
	\mu( \R^d \setminus Q )^{1/p} < \epsilon \big/ \big( 3\cdot 2^{d-1} \max( \norm{g}_{L^q(\mu)}, 1) ).
\]
Consider $\xi \in \R^d$ arbitrary, then we get with the assumed properties of $g$ that
\begin{align}\begin{split}
	| \cF g (\xi) | &\le \left|\int_{\R^d} ( \cos \langle x, \xi \rangle - F_1(x) ) g(x)\, \mu( \intd{x}) \right|\\
	& \quad +  \left|\int_{\R^d} ( \sin \langle x, \xi \rangle - F_2(x) ) g(x)\, \mu( \intd{x}) \right| \label{FourierTransf}
\end{split}\end{align}
for all $F_1, F_2 \in \overline{ \bigcup_{j\in\Z} U_j }$. We show that the first term in \eqref{FourierTransf} is smaller than $\epsilon$ for suitable $F \in \overline{ \bigcup_{j\in\Z} U_j }$; the second term can be treated in the same way. Therefore, we use several times the trigonometric identities	$\sin  = - \cos\left(\darg + \frac{\pi}{2} \right)$, as well as, $\cos( \alpha + \beta) = \cos \alpha \cos \beta - \sin \alpha \sin \beta$: we can split $\cos \langle \darg, \xi \rangle$ in $2^{d-1}$ terms as	$\cos \langle x, \xi \rangle = \sum_{i=1}^{2^{d-1}} b_i \cos( \xi_1 x_1 + a_{i,1}) \cdot \ldots \cdot \cos ( \xi_d x_d + a_{i,d})$, where the $b_i$ are in $\{ -1, 1\}$. Firstly, we prove that each of the functions $\cos(\xi_k\darg + a_{i,k})$ can be uniformly approximated on finite intervals. Indeed, define the kernel
\[
		K\colon \R^2 \rightarrow \R,\, (x,y) \mapsto \sum_{k \in \Z} \phi(x-k)\,\phi(y-k)
\]
and the associated linear wavelet projection operator $K_j$ for $j\in \Z$ by
\[
		K_j\colon  L^2(\lambda) \rightarrow \overline{U_j},\quad f\mapsto \sum_{k\in\Z}  \left \langle f,\; 2^{j/2} \phi(2^j \darg - k) \right\rangle 2^{j/2}\phi (2^j\darg - k).
\]
Then, $K$ satisfies the moment condition $M(N)$ from \cite{hardle2012wavelets} for $N=0$: since $\phi$ is a scaling function, we have	$\int_{\R} K(\darg,y) \intd{y} = \sum_{k\in\Z} \phi(\darg - k) \equiv 1$. Furthermore,
\[
	|K(x,y)| = \left| \sum_{k\in \Z} \phi(x-k)\, \phi(y-k) \right| \le (A+1) \norm{\phi}_{\infty}^2 1_{ \{|x-y| \le A \} } =: F( x-y),
\]
where we assume w.l.o.g.\ that $\overline{\text{supp}\, \phi} \subseteq [0,A]$. Thus, $F$ is integrable w.r.t.\ the Lebesgue measure $\lambda$ and $K$ satisfies the moment condition $M(0)$. Next, let $I(i,k) \supseteq [-An, An]$ be a finite interval such that $\cos( \xi_k \darg + a_{i,k})$ is zero at the boundary of $I(i,k)$. Then by Theorem 8.1 and Remark 8.4 in \cite{hardle2012wavelets} the uniformly continuous restriction $\cos( \xi_k \darg + a_{i,k})\; 1_{I(i,k) }$ can be approximated in $L^{\infty}(\lambda)$ with elements from some $U_j$, i.e.,
\[
		\norm{ \cos( \xi_k \darg + a_{i,k})\, 1_{I(i,k)} - K_j\;\cos( \xi_k \darg + a_{i,k})\, 1_{I(i,k) } }_{L^{\infty}(\lambda)} \rightarrow 0.
\]  
Thus, if $\tilde{\epsilon} > 0$ is arbitrary but fixed, we can choose for each factor $\cos( \xi_k \darg + a_{i,k})\, 1_{ I(i,k) }$ an approximation $f_{i,k}$ in some $U_j$ such that $\norm{ \cos (\xi_k \darg  + a_{i,k} ) 1_{I(i,k)} - f_{i,k} }_{L^{\infty}(\lambda )} \le \tilde{\epsilon}$. This implies that for each of the $i=1,\ldots,2^{d-1}$ products we have
\begin{align}
	&\norm{ \cos( \xi_1 x_1 + a_{i,1}) 1_{ I(i,1)} \cdot \ldots \cdot \cos ( \xi_d x_d + a_{i,d}) 1_{I(i,d)} - f_{i,1}  \otimes \ldots  \otimes f_{i,d} }_{L^{\infty}(\lambda )}  \nonumber \\
	& \qquad \qquad \qquad \qquad  \qquad \qquad  \le (1+\tilde{\epsilon})^d -1\le d \tilde{\epsilon} e^{d\tilde{\epsilon}} \le \left(d e^d\right)\, \tilde{\epsilon}. \label{norm000}
\end{align}
This means that the $d$-dimensional approximation follows from the one-dimensional approximations.

Set now $F_1 \coloneqq \sum_{i=1}^{2^{d-1}} b_i f_{i,1}  \otimes \ldots  \otimes f_{i,d} $ and $\tilde{\epsilon} \coloneqq \epsilon/\left( 3\cdot 2^{d-1} d e^d \norm{g}_{L^q(\mu)} \right)$, then we arrive at
\begin{align}
	&\left|	\int_{\R^d} \left( \cos \scalar{x}{\xi} - F_1(x)	\right) g(x)\,\mu(\intd{x})	\right| \nonumber \\
	&\le \int_{Q}| \cos \scalar{x}{\xi} - F_1(x)| \,|g(x)| \mu(\intd{x}) + \int_{\R^d \setminus Q} | \cos \scalar{x}{\xi} - F_1(x)| \,|g(x)| \mu(\intd{x}) \label{waveletApprox1}
\end{align}
We consider the terms in \eqref{waveletApprox1} separately. We can estimate the first term by
\begin{align}
		\int_{Q}| \cos \scalar{x}{\xi} - F_1(x)| \,|g(x)| \mu(\intd{x})  &\le \sum_{i=1}^{2^{d-1}} \int_Q \left( de^d\right) \tilde{\epsilon} \, |g(x)| \, \mu( \intd{x}) \nonumber \\
		&\le 2^{d-1} d e^d \norm{g}_{L^q(\mu)} \tilde{\epsilon} = \frac{\epsilon}{3}. \label{waveletApprox2}
\end{align}
Likewise, for the second term we infer that
\begin{align}
		&\int_{ \R^d\setminus B} | \cos \scalar{x}{\xi} - F_1(x)| \,|g(x)| \mu(\intd{x}) \nonumber \\
		&\le \sum_{i=1}^{2^{d-1}} \int_{ \R^d \setminus B} \left| \left( \prod_{k=1}^d \cos(\xi_k x_k + a_{i,k} ) \right) 1_{ \times_{k=1}^d I(i,k) } - \prod_{k=1}^d f_{i,k} (x_k) \right| \, |g(x)| \, \mu(\intd{x}) + \ldots \nonumber \\
		&\quad\ldots +  \sum_{i=1}^{2^{d-1}} \int_{ \R^d \setminus B} \left| \left( \prod_{k=1}^d \cos(\xi_k x_k + a_{i,k} ) \right) 1_{ \R^d \setminus \times_{k=1}^d I(i,k) } \right| \, |g(x)| \, \mu(\intd{x}) \nonumber \\
		& \le 2^{d-1} de^d \tilde{\epsilon} \norm{g}_{L^q(\mu)} \mu\left( \R^d\setminus B\right)^{\frac{1}{p}} + 2^{d-1} \norm{g}_{L^q(\mu)}  \mu\left( \R^d \setminus B \right)^{\frac{1}{p}} \nonumber \\
		&= \frac{\epsilon}{3} \cdot \frac{\epsilon}{3 \cdot 2^{d-1} \max(\norm{g}_{L^q(\mu)},1) } + \frac{\epsilon}{3}. \label{waveletApprox3}
\end{align}
All in all, we have when combining \eqref{waveletApprox2} and \eqref{waveletApprox3} that \eqref{waveletApprox1} is less than $\epsilon$ as desired.
\end{proof}

\begin{proof}[Proof of Theorem~\ref{ConsistencyWaveletRegression} and of Theorem~\ref{RateOfConvergenceWaveletRegression}]
We prove that 
$$
		\inf_{ f\in\cF_n, \norm{f}_{\infty} \le \rho_n  } \int_{\R^d} \left|f-m\right|^2\,\intd{\mu_X} \rightarrow 0.
		$$
Let $\epsilon > 0$. Since $\bigcup_{j\in\N} U_j$ is dense in $L^2(\mu_X)$, there is a function $f$ and a $j_0\in\N$ such that for all $j\ge j_0$, we have $f\in U_{j}$ and $\int_{\R^d} |f-m|^2 \,\intd{\mu_X} < \epsilon/4$. We can write for each level $j(n)$
\[
		f = \sum_{\gamma\in K_n} a_{j,\gamma}\, \Psi_{j,\gamma} + \sum_{\gamma\notin K_n} a_{j,\gamma}\, \Psi_{j,\gamma}
\]
for coefficients $a_{j,\gamma}\in\R$. Set $g_n \coloneqq \sum_{\gamma\notin K_n} a_{j,\gamma}\, \Psi_{j,\gamma}$. The support of the $g_n$ decreases monotonically to zero:
\begin{align}
		\{ g_n \neq 0 \} &\subseteq \left\{ x\in \R^d: M^j x - \gamma \in [0,L]^d,\, \norm{\gamma}_{\infty} > w_n \right\} \nonumber \\
				&\subseteq \left\{ x\in \R^d: \norm{ M^j x }_{\infty} \ge \norm{\gamma}_{\infty} - L,\, \norm{\gamma}_{\infty} > w_n \right\}  \nonumber \\
				&\subseteq \left\{ x\in \R^d: \norm{ M^j x }_2 \ge w_n - L \right\}  \nonumber \\
		&\subseteq \left\{ x\in \R^d: \norm{S^{-1}}_2 (\zeta_{max})^j \norm{S}_2 \norm{x }_2 \ge w_n - L \right\}  \downarrow \emptyset \quad (n\rightarrow \infty), \nonumber
\end{align}
by the assumption that $(\zeta_{max})^{j} / w_n \rightarrow 0$ as $n \rightarrow \infty$. Furthermore, there is a $k_1\in \N$ such that we have the estimate $\int_{\R^d}f^2\, \1{ \R^d\setminus [-k_1,k_1]^d } \intd{\mu_X} < \epsilon/4$ for all $k\ge k_1$. Hence, there is a $k_2\in \N$ such that both
\[
	[-k_1,k_1]^d \subseteq \bigcup_{\gamma\in K_{n} } \supp{ \Psi_{j,\gamma}} \text{ and } \norm{f\, \1{ [-k_1,k_1]^d } }_{\infty} \le \rho_n
\]
for all $k\ge k_2$. In particular, the function $f\, \1{[-k_1,k_1]^d }$ is admissible in the sense that it is in $T_{\rho_n} \cF_n$ and that $\int_{\R^d} |m- f\, \1{[-k_1,k_1]^d } |^2 \, \intd{\mu_X} < \epsilon$ as desired.

For the second part, it remains to compute $\kappa_n(\epsilon,\rho_n) = \log H_{T_{\rho_n } \cF_n}\left(  \epsilon/(4\rho_n ) \right)$. We use the bound which is given in Proposition~\ref{boundCoveringNumber}, we have
\begin{align*}
	&H_{T_{\rho_n} \cF_n } \left(  \epsilon/(4\rho_n ) \right) \le 3 \exp\left\{ 2( (2w_n+1)^d+1) \log( 384 e\, \rho_n^2/\epsilon ) \right\}, \\
	&\qquad\qquad\qquad\qquad\qquad\qquad\qquad\qquad\qquad \text{ i.e., } \kappa_n(\epsilon,\rho_n) = \cO\left( w_n^d \log( \rho_n )\right)
\end{align*}
for $\epsilon > 0$ which is arbitrary but fixed. The statement concerning the consistency properties follows now from Theorem~\ref{ConsistencyTruncatedLeastSquaresGeneral}. The statement which concerns the rate of convergence follows from Theorem~\ref{nonParaRegRateOfConvergence}.
\end{proof}

\section*{Acknowledgement}
The author is very grateful to two referees and an associate editor, their comments and suggestions greatly improved the manuscript.

\appendix

\section{Exponential Inequalities for Dependent Sums}\label{Section_ExponentialInequalities}

In this section, we give a short review on important concepts which we use throughout the article. We begin with a result concerning the covering number of a function class $\cG$ of real-valued functions on $\R^d$. Denote the class of all subgraphs of this class $\cG$ by $\cG^+ \coloneqq \big\{ \big\{ (z,t)\in \R^d\times\R: t\le g(z) \big\} : g\in \cG \big\}$ and the Vapnik-Chervonenkis-dimension of $\cG^+$ by $\cV_{\cG^+}$. In this case Condition \ref{coveringCondition} is satisfied if $\epsilon$ is sufficiently small and if the Vapnik-Chervonenkis dimension of $\cG^+$ is at least two. More precisely, we have the following statement

\begin{proposition}[\cite{haussler1992decision}]\label{boundCoveringNumber}
Let $[a,b] \subseteq \R$. Let $\cG$ be a class of uniformly bounded real valued functions $g\colon  \R^d \to [a,b]$ such that $\cV_{\cG^+} \ge 2$. Let $0 < \epsilon < (b-a)/4$. Then for any probability measure $\nu$ on $(\R^d,\cB(\R^d))$
\begin{align*}
		\log \mathsf{N}\left( \epsilon, \cG, \norm{\,\cdot\,}_{L^p(\nu)} \right) \le \log 3 + \cV_{\cG^+} \log\left( \frac{2e (b-a)^p}{\epsilon^p} \,\log\frac{3e(b-a)^p}{\epsilon^p} \right) .
\end{align*}
In particular, $\cV_{\cG^+} \le r+1$ in the case where $\cG$ is an $r$-dimensional linear space.
\end{proposition}

Davydov's inequality relates the covariance of two random variables to the $\alpha$-mixing coefficient:
\begin{proposition}[\cite{davydov1968convergence}]\label{davydovsIneq}
Let $\pspace$ be a probability space and let $\cG, \cH$ be sub-$\sigma$-algebras of $\cA$. Set $\alpha \coloneqq \sup\{ |\p(A \cap B) - \p(A)\p(B) |:\, A \in \cG, B\in \cH \}$. Let $p,q,r \ge 1$ be H{\"o}lder conjugate, i.e., $p^{-1} + q^{-1} + r^{-1} =1$. Let $\xi$ (resp. $\eta$) be in $L^p(\p)$ and $\cG$-measurable (resp. in $L^q(\p)$ and $\cH$-measurable). Then $\left| \text{Cov}(\xi,\eta) \right|  \le 10\, \alpha^{1/r} \norm{\xi}_{L^p(\p)} \norm{\eta}_{L^q(\p)}$.
\end{proposition}

In the remaining part of this appendix we derive upper bounds on the probability of events of the type
\begin{align}\label{supMeas1}
		\left \{ \sup_{g\in\cG} \left||I_n|^{-1} \sum_{s \in I_n} g(Z(s)) - \E{g(Z(e_N) )} \right| > \epsilon \right  \},
\end{align}
where $\cG$ is a class of functions and $Z$ is an $\alpha$-mixing random field on $\Z^N$. (Here we assume that $\cG$ is sufficiently regular, so that this set is indeed measurable.)

\begin{proposition}\label{applBernstein}
Let the real valued random field $Z$ satisfy Condition \ref{regCond0} \ref{Cond_AlphaMixing}. The $Z(s)$ have expectation zero and are bounded by $B$. Assume that $n\in \N_+^N$ satisfies
$ \min\{ n_i: i=1,\ldots,N \} / \max\{ n_i: i=1,\ldots,N \}  \ge C'$,
for a constant $C'> 0$. Then there is a constant $A\in \R_+$ which depends on the lattice dimension $N$, the constant $C'$ and the bound on the mixing coefficients but neither on $n\in\N_+^N$, nor on $\epsilon$, nor on $B$ such that for all $\epsilon>0$
\begin{align*}
		\p\Big( |I_n|^{-1} \Big| \sum_{s\in I_n} Z(s) \Big| \ge \epsilon \Big) \le \exp\left(-\frac{A |I_n|^{1/N} \epsilon^2 }{B^2 + B \epsilon  (\log |I_n| )(\log \log |I_n|)}	\right).
\end{align*}
\end{proposition}
\begin{proof}
One can apply the exponential inequality of \cite{merlevede2009bernstein} for strongly mixing time series to the random field $Z$ as follows: consider a fixed $n\in\N^N$ and set $j^* = \min\{1\le j\le N: n_j = \min\{ n_i: i=1,\ldots,N\} \}$. Then define a time series $Y$ by
$$
		Y_k = \sum_{s\in I_n, s_{j^*} = k } Z_s, \quad k=1,\ldots, n_{j^*}.
$$
We have that the $Y_k$ are bounded by $B (C' n_{j^*} )^{N-1}$. This time series is strongly mixing with exponentially decreasing mixing coefficients (in the sense of the weaker definition for time series, cf. \cite{doukhan1991mixing}). The result follows now from Theorem 1 of \cite{merlevede2009bernstein}.
\end{proof}

We can prove with the previous proposition an important statement

\begin{theorem}\label{USLLNM}
Assume that the conditions of Proposition~\ref{applBernstein} are satisfied. Let $\cG$ be a set of measurable functions $g\colon \R^d \rightarrow [0,B]$ for $B\in [1,\infty)$ which satisfies Condition \ref{coveringCondition} and assume that \eqref{supMeas1} is measurable. Then there is a constant $A$ which is independent of $\epsilon$, $n$ and $B$ such that for all $\epsilon>0$
\begin{align}
\begin{split}\label{USLLNMEq0}
&\p\left( \sup_{g \in \cG} \left| |I_n|^{-1} \sum_{ s \in I_n} g(Z(s)) - \E{ g(Z(e_N) )}		\right| \ge \epsilon		\right)  \\
&\quad\le 
10 H_{\cG}\left( \frac{\epsilon}{32} \right) \left\{	\exp\left( - \frac{|I_n| \epsilon^2}{512 B^2} \right) + \exp\left(-\frac{A |I_n|^{1/N} \epsilon^2 }{B^2 + B \epsilon  (\log |I_n| )^2 }	\right)	 \right\}.
\end{split}
\end{align}
\end{theorem}

\begin{proof}[Proof of Theorem \ref{USLLNM}]
We assume the probability space to be endowed with the i.i.d.\ random variables $Z'(s)$ for $s \in I_n$ which have the same marginal laws as the $Z(s)$. We write for shorthand
\[
		S_n(g) := \frac{1}{ |I_n|} \sum_{s\in I_n} g(Z(s)) \text{ and } S'_n(g) := \frac{1}{|I_n|} \sum_{s\in I_n} g(Z'(s)).
\]
We can decompose the probability with these definitions as follows
\begin{align}
	&\p\left(	\sup_{g\in\cG} \left| S_n(g) - \E{ g(Z(e_N) )} \right| \ge \epsilon	\right) \nonumber\\
		&\le \p\left(	\sup_{g \in \cG} \left| S_n(g) - S'_n(g) \right| \ge \frac{\epsilon}{2}\right) + \p\left(	\sup_{g \in \cG} \left| S'_n(g) - \E{ g(Z'(e_N) )} \right| \ge  \frac{\epsilon}{2}\right). \label{USLLNM1}
\end{align}
Next, we apply Theorem 9.1 from \cite{gyorfi} to second term on the right-hand side of \eqref{USLLNM1} and obtain
\begin{align}
		\p\left(	\sup_{g \in \cG} \left| S'_n(g) - \E{ g(Z'(e_N) )} \right| \ge \frac{\epsilon}{2}\right) \le 8 H_{\cG} \left( \frac{\epsilon}{16} \right) \exp\left( - \frac{|I_n| \epsilon^2}{512 B^2} \right). \label{USLLNM1b} 
\end{align}
To get a bound on the first term of the right-hand side of \eqref{USLLNM1}, we use Condition \ref{coveringCondition} to construct an $\epsilon/32$-covering. Write $H^*:=H_{\cG} \left( \frac{\epsilon}{32} \right)$ for the upper bound on the covering number. Let $g_k^{\ast}$ for $k=1,\ldots,H^*$ be as in Condition \ref{coveringCondition}. Define
\[
		U_k := \left\{ g\in \cG: \frac{1}{2|I_n|} \sum_{s\in I_n} \Big|g(Z(s))-g^{\ast}_k(Z(s))\Big| + \Big|g(Z'(s)) - g^{\ast}_k(Z'(s))\Big| \le \frac{\epsilon}{32}		\right\}.
\]
Then
\begin{align}
		&\p\left(	\sup_{g \in \cG} \left| S_n(g) - S'_n(g) \right| \ge  \frac{\epsilon}{2}\right)  \le \sum_{k=1}^{H^*} \p\left(	\sup_{g \in U_k} \left| S_n(g) - S'_n(g) \right| \ge  \frac{\epsilon}{2}\right). \label{USLLNM2}
\end{align}
Thus, using the approximating property of the functions $g^*_k$, we get for each probability in \eqref{USLLNM2}
\begin{align}
	&\p\left(	\sup_{g \in U_k} \left| S_n(g) - S'_n(g) \right| \ge  \frac{\epsilon}{2}\right) \nonumber\\
	&\quad \le \p\left( \left| S_n(g^*_k) - S'_n(g^*_k) \right|\ge  \frac{7\epsilon}{16} 	\right) \nonumber \\
	\begin{split}
	&  \le \p\left( \left|S_n(g^*_k) - \E{ g^*_k(Z(e_N) )} \right|	\ge  \frac{7\epsilon}{32} 	\right) \\
	&\quad + 
	 \p\left( \left|S'_n(g^*_k) - \E{ g^*_k(Z'(e_N) )} \right|	\ge \frac{7\epsilon}{32} 	\right). \label{USLLNM4}
	\end{split}
\end{align}
The second term on the right-hand side of \eqref{USLLNM4} can be estimated using Hoeffding's inequality, we have
\begin{align*}
		\p\left( \left|S'_n(g^*_k) - \E{ g^*_k(Z'(e_N))} \right|	\ge \frac{7\epsilon}{32} 	\right) \le 2 \exp \left\{ - \frac{98 \, |I_n| \, \epsilon^2}{32^2 \, B^2}		\right\}. 
\end{align*}
We apply Proposition \ref{applBernstein} to the first term of \eqref{USLLNM4}. Finally, we use that $H_{\cG} \left(\frac{\epsilon}{16} \right) \le H_{\cG} \left(\frac{\epsilon}{32} \right)$. 
\end{proof}

\section{Details on Example~\ref{MRAByTensorProduct}}\label{Appendix_Example}
We show how to derive an isotropic MRA in $d$ dimensions from a one-dimensional MRA. It is straightforward to show that for a multiresolution analysis with corresponding scaling function $\Phi$ there is a sequence $( a_0(\gamma): \gamma \in \Gamma) \subseteq \R$ such that $\Phi \equiv \sum_{\gamma \in \Gamma} a_0(\gamma)\, \Phi(M\darg - \gamma)$ and the coefficients $a_0(\gamma)$ satisfy the equations $a_0(\gamma) = |M| \int_{\R^d} \Phi(x)\, \Phi(Mx-\gamma)\, \intd{x}$ as well as $\sum_{\gamma\in\Gamma} |a_0(\gamma)|^2 = |M| = \sum_{\gamma\in \Gamma} a_0(\gamma)$.

In the first step, we show that the conditions for an MRA are satisfied. The spaces $\bigcup_{j\in\Z} U_j$ are dense: we have by the definition
\[
	U_j = \bigotimes_{i=1}^{d} U'_j = \left\langle	f_1\otimes\ldots\otimes f_d: f_i \in U'_j \; \forall i=1,\ldots,d	\right\rangle.
\]
Note that the set of pure tensors $\left\langle g_1\otimes \ldots \otimes g_d: g_i \in L^2(\lambda) \right\rangle$ is dense in $L^2(\lambda^d)$. Hence, it only remains to show that we can approximate any pure tensor $g_1\otimes\ldots\otimes g_d$ by a sequence $(F_j\in U_j: j \in \N_+)$. Let $\epsilon > 0 $ and let $g_1\otimes \ldots \otimes g_d \in L^2(\lambda^d)$ be a a pure tensor. Choose a sequence of pure tensors $(f_{i,j}: j \in \N_+ )$ converging to $g_i$ in $L^2(\lambda)$ for $i=1,\ldots,d$. Denote by $L\coloneqq \sup\left\{ \norm{f_{i,j}}_{L^2(\lambda)}, \norm{g_i}_{L^2(\lambda)}: j \in \Z, i=1,\ldots,d \right\} < \infty.$ Then
\begin{align*}
		&\norm{ g_1\otimes \ldots\otimes g_d - f_{1,j} \otimes \ldots \otimes f_{d,j} }^2_{L^2(\lambda^d) } \\
		&\le d^2 L^{2(d-1) } \max_{1\le i \le d} \norm{ g_i - f_{i,j}}^2_{L^2(\lambda)} \rightarrow 0 \text{ as } j \rightarrow \infty.
\end{align*}
Furthermore, $\bigcap_{j\in \Z} U_j = \{0\}$: Let $f = \sum_{i=1}^n a_i \,f_{i,1} \otimes \ldots \otimes f_{i,d}$ be an element of each $U_j$. Then each $f_{i,k}$ is an element of each $U'_j$ for all $j$ and, hence, zero.
The scaling property is immediate, too. Indeed,
\begin{align*}
	&f \in U_j \Leftrightarrow f = \sum_{i=1}^n a_i f_{i,1} \otimes \ldots \otimes f_{i,d} \text{ and } f_{i,k} \in U'_j, \quad k = 1,\ldots,d \\
	&\qquad \qquad \Leftrightarrow f = \sum_{i=1}^n a_i f_{i,1} \otimes \ldots \otimes f_{i,d} \text{ and } f_{i,k}(2^{-j} \darg) \in U'_0 \Leftrightarrow f(M^{-j}\darg) \in U_0.
\end{align*}
The functions $\{\Phi(\darg - \gamma): \gamma \in \Gamma \}$ form an orthonormal basis of $U_0$. We have for $\gamma, \gamma' \in \Z^d$
\begin{align*}
		&\int_{\R^d} \Phi(x-\gamma)\, \Phi(x-\gamma')\, \intd{x} = \int_{\R^d} \otimes_{k=1}^{d} \phi(x_k-\gamma_k) \cdot \otimes_{k=1}^{d} \phi(x_k-\gamma'_k) \,\intd{x} \\
		&= \prod_{k=1}^{d} \int_{\R} \phi(x_k-\gamma_k)\phi(x_k-\gamma'_k)\,\intd{x_k} = \delta_{\gamma, \gamma'}
\end{align*}
and for each $f\in U_0$ by definition $f=\sum_{i=1}^{n} a_i\, \phi(\darg - \gamma_{1}^{i} ) \cdot\ldots\cdot \phi(\darg - \gamma_{d}^{i}) = \sum_{i=1}^n a_i \Phi( \darg - \gamma^{i} )$ for $\gamma^{1},\ldots,\gamma^{n}\in \Z^d$. This proves that $\Phi$ together with the linear spaces $U_j$ generates an MRA of $L^2(\lambda^d)$. It remains to prove that the wavelets generate an orthonormal basis of $L^2(\lambda^d)$.

For an index $k\in \bigtimes_{i=1}^d \{0,1\}$, define $a^{k_i}_l$ by $\sqrt{2} h_l$ if $k_i=0$ and $\sqrt{2} g_l$ if $k_i = 1$ for $i=1,\ldots,d$. Furthermore, set $a_k(\gamma) \coloneqq a^{k_1}_{\gamma_1} \cdot\ldots\cdot a^{k_d}_{\gamma_d}$. Then, the scaling function and the wavelet generators satisfy
\[
		\Psi_k = \sum_{\gamma_1,\ldots,\gamma_d} a^{k_1}_{\gamma_1} \cdot\ldots\cdot a^{k_d}_{\gamma_d}\; \phi(2\darg - \gamma_1) \otimes\ldots\otimes \phi(2\darg - \gamma_d) = \sum_{\gamma} a_k(\gamma) \Phi(M\darg - \gamma).
\]
Since $\phi$ is a scaling function, the coefficients $a_0(\gamma)$ of the scaling function $\Phi$ satisfy the relation
\begin{align*}
		\sum_{\gamma} a_0(\gamma) &= 2^{d/2} \sum_{\gamma_1,\ldots,\gamma_d} h_{\gamma_1}\cdot\ldots\cdot h_{\gamma_d} = 2^{d/2} \left(\sum_{\gamma_1} h_{\gamma_1} \right)^d = 2^d.
\end{align*}
Furthermore, we have for $j,k \in \{0,1\}^d$ and $\gamma\in \Gamma$,
\[
		\sum_{\gamma'} a_j(\gamma') a_k(M\gamma+\gamma') = \left\{ \sum_{\gamma'_1} a^{j_1}_{\gamma'_1} a^{k_1}_{2 \gamma_1 + \gamma'_1}	\right\} \cdot\ldots\cdot \left\{ \sum_{\gamma'_d} a^{j_d}_{\gamma'_d} a^{k_d}_{2 \gamma_d + \gamma'_d}	\right\} = 2^d \delta_{j,k} \delta_{\gamma,0}.
\]
Indeed, for $s=1,\ldots,d$ and $z\coloneqq\gamma_s$
\[
		\sum_{\gamma'_s} a^{j_s}_{\gamma'_s} a^{k_s}_{2 \gamma_s + \gamma'_s} = \begin{cases}
			2 \sum_{l} h_l g_{2 z + l} & \quad \text{ if } j_s = 0 \text{ and } k_s = 1, \\
			2 \sum_{l} h_l h_{2 z + l} & \quad \text{ if } j_s = k_s = 0, \\
			2 \sum_{l} g_l h_{2 z + l} & \quad \text{ if } j_s = 1 \text{ and } k_s = 0, \\
			2 \sum_{l} g_l g_{2 z + l} & \quad \text{ if } j_s = k_s = 1. \\
		\end{cases}
\]
Since the $\phi(\darg - z)$ form an ONB of $U'_0$, we have
\[
		\delta_{z,0} = \int_{\R} \phi(x-z)\, \phi(x)\; \intd{x} =  \sum_{l,m} h_l h_m \delta_{2z+l,m} =  \sum_l h_l h_{2z+l}.
\]
In the same way,
\[
	\delta_{z,0} = \int_{\R} \psi(x-z) \psi(x)\, \intd{x} =  \sum_{l,m} g_l g_m \delta_{2z + l,m} =  \sum_l g_l g_{2z+l}.
\]
In addition, as $U'_1 = U'_0 \otimes W'_0$, we get
\[
		0 = \int_{\R} \psi( x - z) \, \phi(x)\; \intd{x} =  \sum_{l,m} g_l h_m \delta_{2z+l,m} =  \sum_l g_l h_{2z+l} =  \sum_l g_{l-2z} h_l,
\]
for all $z\in\Z$. Hence, the conditions of Theorem~\ref{BenedettoTheorem} (Theorem 1.7 in \cite{benedetto1993wavelets}) are satisfied and the family of functions $\{ |M|^{j/2} \Psi_k (M^j\darg - \gamma): \gamma \in \Gamma, k=1,\ldots,|M|-1 \}$ forms an ONB of $W_j$ and $L^2(\lambda^d) = \bigoplus_{j\in\Z} W_j$.

\end{document}